\documentclass[12pt]{amsart}

\usepackage[utf8]{inputenc}
\usepackage[a4paper,nomarginpar]{geometry}
\usepackage[english]{babel}
\usepackage{enumitem}
\usepackage{amscd,amsfonts,amsmath,amssymb,amsthm,latexsym}
\usepackage{color,colortbl,xcolor,graphicx,geometry}
\usepackage{pgf,tikz,pgfplots}
\usepackage[all]{xy}

\allowdisplaybreaks

\makeatletter
\def\subsection{\@startsection{subsection}{2}%
    \z@{.5\linespacing\@plus.7\linespacing}{.3\linespacing}%
    {\normalfont\bfseries}}
\makeatother

\theoremstyle{theorem}
\newtheorem{theorem}{Theorem}
\newtheorem{lemma}[theorem]{Lemma}
\newtheorem{proposition}[theorem]{Proposition}
\newtheorem{corollary}[theorem]{Corollary}

\theoremstyle{definition}
\newtheorem{definition}[theorem]{Definition}
\newtheorem{example}[theorem]{Example}
\newtheorem{remark}[theorem]{Remark}

\theoremstyle{remark} \theoremstyle{question} \theoremstyle{example}

\newcommand{\N}{\mathbb{N}}   
\newcommand{\Z}{\mathbb{Z}}   
\newcommand{\R}{\mathbb{R}}   
\newcommand{\C}{\mathbb{C}}   
\newcommand{\K}{\mathbb{K}}   

\newcommand{\cC}{\mathcal{C}}

\newcommand{\cU}{\mathcal{U}}

\newcommand{\eps}{\varepsilon}   
\newcommand{\ov} {\overline}     
\newcommand{\wt} {\widetilde}    

\newcommand{\CR}{\operatorname{CR}}
\newcommand{\Orb}{\operatorname{Orb}}
\newcommand{\Per}{\operatorname{Per}}

\newcommand{\Ker}{\operatorname{Ker}}

\begin{document}

\title[Generalized hyperbolicity, stability and expansivity for operators on LCS]
{Generalized hyperbolicity, stability and expansivity\\ for operators on locally convex spaces}

\subjclass[2020]{Primary 47A16, 37B65, 37B25; Secondary 37B05, 37C50, 37D20.}
\keywords{Generalized hyperbolicity, shadowing property, topological stability, expansivity, Li-Yorke chaos, weighted shifts}
\date{}
\dedicatory{}
\maketitle

\begin{center}
{\sc Nilson C. Bernardes Jr., \ Blas M. Caraballo, \ Udayan B. Darji,\\ Vin\'icius V. F\'avaro \ and \ Alfred Peris}
\end{center}

\bigskip

\begin{abstract}
We introduce and study the notions of (generalized) hyperbolicity, topological stability and (uniform) topological expansivity
for operators on locally convex spaces.
We prove that every generalized hyperbolic operator on a locally convex space has the finite shadowing property.
Contrary to what happens in the Banach space setting, hyperbolic operators on Fr\'echet spaces may fail to have the shadowing property,
but we find additional conditions that ensure the validity of the shadowing property.
Assuming that the space is sequentially complete, we prove that generalized hyperbolicity implies the
strict periodic shadowing property, but we also show that the hypothesis of sequential completeness is essential.
We show that operators with the periodic shadowing property on topological vector spaces have other interesting dynamical behaviors,
including the fact that the restriction of such an operator to its chain recurrent set is topologically mixing and Devaney chaotic.
We prove that topologically stable operators on locally convex spaces have the finite shadowing property
and the strict periodic shadowing property.
As a consequence, topologically stable operators on Banach spaces have the shadowing property.
Moreover, we prove that generalized hyperbolicity implies topological stability for operators on Banach spaces.
We prove that uniformly topologically expansive operators on locally convex spaces are neither Li-Yorke chaotic
nor topologically transitive.
Finally, we characterize the notion of topological expansivity for weighted shifts on Fr\'echet sequence spaces.
Several examples are provided.
\end{abstract}


\section{Introduction}

The {\em shadowing property} is one of the most important concepts in the modern theory of dynamical systems.
It originated with works by Anosov, Bowen and Sina$\breve{\text{\i}}$ from the late 1960s and early 1970s,
leading to the famous {\em shadowing lemma} in differentiable dynamics, which asserts that a diffeomorphism
has the shadowing property in a neighborhood of its hyperbolic set \cite{DAno67,RBow75,JSin72}.
In the setting of linear dynamics, it is well known that every invertible hyperbolic operator on a Banach space
has the shadowing property in the full space \cite{AMor81,JOmb94}.
Moreover, the converse holds in the finite dimensional setting \cite{AMor81,JOmb94} and for invertible normal
operators on Hilbert spaces \cite{MMaz00}.
It remained open for a while whether or not this converse is always true.
The solution was given by means of the following theorem (where $r(T)$ denotes the spectral radius of $T$):
\textit{ Let $T$ be an invertible operator on a Banach space $X$. Suppose that $X = M \oplus N$,
where $M$ and $N$ are closed subspaces of $X$ with $T(M) \subset M$ and $T^{-1}(N) \subset N$.
If $r(T|_M) < 1$ and $r(T^{-1}|_N) < 1$, then $T$ has the shadowing property} \cite[Theorem~A]{BerCirDarMesPuj18}.

The class of operators considered in the above theorem clearly contains all invertible hyperbolic operators,
but it also enabled the construction of the first examples of operators that have the shadowing property but are not
hyperbolic \cite[Theorem~B]{BerCirDarMesPuj18}, thereby solving the above-mentioned problem in the negative.
The operators considered in the above theorem were named {\em generalized hyperbolic} in \cite{CirGolPuj21},
where additional dynamical properties of these operators where investigated.
The class of generalized hyperbolic operators is also closely related to another fundamental concept in dynamical systems,
namely: {\em structural stability}.
A classical theorem in linear dynamics, often called {\em Hartman's theorem}, asserts that every invertible hyperbolic opera\-tor
on a Banach space is structurally stable \cite{PHar60,JPal68,CPug69}.
It was soon realized that the converse holds in the finite dimensional setting \cite{JRob72}, but whether or not the converse
of Hartman's theorem is always true remained open for more than 50 years.
This problem was finally settled in \cite{NBerAMes21}, where it was obtained a class of weighted shifts on classical Banach
sequence spaces that are structurally stable but are not hyperbolic \cite[Theorem~9]{NBerAMes21}.
It turns out that these weighted shifts are generalized hyperbolic.
A little later, it was proved that every generalized hyperbolic operator is structurally stable \cite[Theorem~1]{NBerAMes20},
which enabled the proof of a {\em generalized Grobman-Hartman theorem} \cite[Theorem~3]{NBerAMes20}.

Another important concept in dynamical systems is that of {\em topological stability}, which was introduced by Walters~\cite{PWal1970}.
A classical result due to Walters~\cite{PWal1978} asserts that: \textit{ Every topologically stable homeomorphism $h : X \to X$,
where $X$ is a closed topological manifold of dimension at least two, has the shadowing property}
(see also \cite[Theorem~2.4.9]{NAokKHir94}).
Recall that a {\em closed topological manifold} is a compact connected metrizable topological manifold without boundary
\cite[Page~28]{NAokKHir94}.
Motivated by this result, we will begin a study of the concept of topological stability in the setting of linear dynamics,
with emphasis on its connections with generalized hyperbolicity and the shadowing property.

Yet another fundamental concept in dynamical systems is that of {\em expansivity}, which was introduced by Utz \cite{WUtz50}.
Expansive and uniformly expansive operators on Banach spaces were studied in
\cite{AlvBerMes21,BerCirDarMesPuj18,NBerAMes21,MEisJHed70,JHed71,MMaz00}, for instance.
Let us mention three of the main results obtained so far:
\begin{itemize}
\item \textit{ An invertible operator on a Banach space is uniformly expansive if and only if its approximate point spectrum
  does not intersect the unit circle} \cite{MEisJHed70,JHed71}.
\item \textit{ An invertible operator on a Banach space is hyperbolic if and only if it is expansive and has the shadowing property}
  \cite{NBerAMes21,CirGolPuj21}.
\item \textit{ A uniformly expansive operator on a Banach space is never Li-Yorke chaotic} \cite{BerCirDarMesPuj18}.
\end{itemize}
Moreover, complete characterizations of expansive and uniformly expansive weighted shifts on classical Banach sequence spaces were obtained in \cite{BerCirDarMesPuj18}.

All the works on linear dynamics mentioned above deal with operators on Banach spaces, but the first results on linear dynamical systems go back to Birkhoff \cite{GBir29} for the translation operator on the (non-normable) Fr\'echet space $H(\C)$ of all entire functions, and to MacLane \cite{GMac52} for the differentiation operator on $H(\C)$, which are the classical examples of chaotic operators.
Moreover, even the dynamics of operators on non-metrizable topological vector spaces has attracted the interest of many researchers and experienced a great development in recent years (see, e.g., Chapter 12 in \cite{KGroAPer11}, the articles \cite{Bo00,BoDo12,BoFrPeWe2005,BoKalPe2021,BCarVFav20,DoKa18,GEPe10,Peris18,Shkarin2012}, and references therein).
Our main objective in the present work is to propose a concept of (generalized) hyperbolicity, a notion of topological stability and
a concept of (uniform) expansivity for operators on locally convex spaces, with the purpose to initiate investigations on these notions. 

A diagram summarizing the known implications (arrows with number (1)) and the implications that will be proved in the present
article (arrows with number (2)), for invertible operators in the Banach setting, is presented in Figure~\ref{diag}.

\begin{figure}[h]\label{diag}
\includegraphics[width=0.8\textwidth]{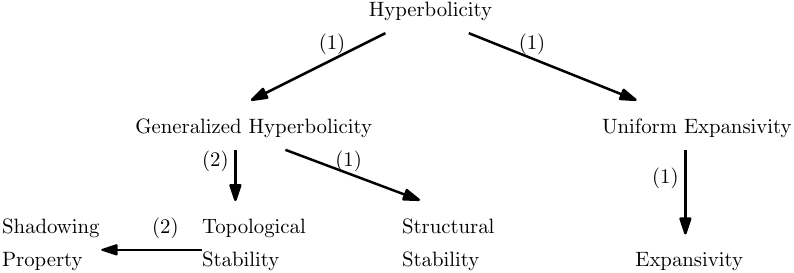}
\caption{Implications between various properties for invertible operators on Banach spaces}
\end{figure}

In Section~\ref{GHSection} we introduce a concept of {\em hyperbolicity} and a concept of {\em generalized hyperbolicity}
for operators on an arbitrary locally convex space (Definition~\ref{gh-def}).
We prove that every generalized hyperbolic operator on a locally convex space has the finite shadowing property (Theorem~\ref{GHFSP}),
but the shadowing property may fail even for hyperbolic operators on the Fr\'echet space $H(\C)$ of all entire functions
(Remark~\ref{NotShad}).
Nevertheless, we give additional conditions to ensure the validity of the shadowing property (Theorem~\ref{GHSP}).
For sequentially complete locally convex spaces, we prove that generalized hyperbolicity implies the
strict periodic shadowing property (Theorem~\ref{GHFSP}), but we also show that the hypothesis of sequential completeness
is essential for the validity of this result (Remark~\ref{CounterexCompleteness}).
We show that operators with the periodic shadowing property on topological vector spaces have other interesting dynamical behaviors
(Theorem~\ref{PSPEVT}), including the fact that the restriction of such an operator to its chain recurrent set is topologically mixing
and Devaney chaotic.
Some examples illustrating the theorems are also presented.

In Section~\ref{TSSection} we prove that every topologically stable operator on a locally convex space has the
finite shadowing property and the strict periodic shadowing property (Theorem~\ref{TSO-FSP}) and that
every topologically stable operator on a Banach space has the shadowing property (Corollary~\ref{TSO-FSP-Cor}).
For this purpose, we show that open convex sets in locally convex spaces of dimension greater than two have
a certain multihomogeneity property (Theorem~\ref{LCS-USLMH}) and we establish a very general version of the previously mentioned
theorem of Walters~\cite{PWal1978} in the setting of uniform spaces (Theorem~\ref{P-FSP}).
Moreover, we show that every invertible generalized hyperbolic operator on a sequentially complete locally convex space
has a certain stability property (Theorem~\ref{GH-TS}).
As a consequence, we obtain that every invertible generalized hyperbolic operator on a Banach space is topologically stable
(Corollary~\ref{GH-TS-Cor}).

In Section~\ref{ExpSection} we introduce the notions of {\em topological expansivity} and {\em uniform topological expansivity}
for invertible operators on an arbitrary locally convex space (Definition~\ref{exp-uexp}).
These concepts generalize the well-known notions of expansivity and uniform expansivity for operators on normed spaces,
but they seem to be more adequate when we go beyond the normed space setting.
For invertible generalized hyperbolic operators on locally convex spaces, the notions of topological expansivity,
uniform topological expansivity and hyperbolicity coincide (Theorem~\ref{GHEquiv}).
In particular, every invertible hyperbolic operator on a locally convex space is uniformly topologically expansive (Corollary~\ref{h-ute}).
Our main result in this section asserts that uniformly topologically expansive operators on locally convex spaces
are neither Li-Yorke chaotic nor topologically transitive (Theorem~\ref{ute-notLY}).

In Section~\ref{TEWS} we characterize the notion of topological expansivity for weighted shifts on Fr\'echet sequence spaces
(Theorem~\ref{shift_space_w}) and, in particular, on K\"othe sequence spaces (Corollary~\ref{applkothe}).
Several concrete examples are presented.

In our final Section~\ref{FinalSection} we propose some open problems.


\section{Generalized hyperbolic operators on locally convex spaces}\label{GHSection}

Throughout $\K$ denotes either the field $\R$ of real numbers or the field $\C$ of complex numbers,
$\N$ denotes the set of all positive integers and $\N_0\!:= \N \cup \{0\}$.

Given a topological vector space $X$ over $\K$ (all topological vector spaces will be assumed to be Hausdorff),
we denote by $L(X)$ the set of all continuous linear operators on $X$
and by $GL(X)$ the set of those operators that have a continuous inverse.
Recall that $X$ is the {\em topological direct sum} of the subspaces $M_1,\ldots,M_n$ if $X$ is the algebraic direct sum
of $M_1,\ldots,M_n$ and the canonical algebraic isomorphism
\[
(y_1,\ldots,y_n) \mapsto y_1 + \cdots + y_n
\]
is a homeomorphism from the product space $M_1 \times \cdots \times M_n$ onto $X$.
Recall also that a family $(\|\cdot\|_\alpha)_{\alpha \in I}$ of seminorms on $X$ is said to be {\em directed} if
for every $\alpha,\beta \in I$, there exists $\gamma \in I$ such that
$\|\cdot\|_\alpha \leq \|\cdot\|_\gamma$ and $\|\cdot\|_\beta \leq \|\cdot\|_\gamma$.

Our first goal is to introduce a notion of generalized hyperbolicity for operators on locally convex spaces.
It is well known that the concept of {\em spectrum} of an operator does not behave so well when we go beyond the Banach space setting.
In fact, even on Fr\'echet spaces, the spectrum of an operator may fail to be a bounded set, so that we loose the concept of 
{\em spectral radius} and the extremely useful {\em spectral radius formula}. Thus, the conditions
\[
r(T|_M) < 1 \ \ \text{ and } \ \ r(T^{-1}|_N) < 1,
\]
that appear in the definition of generalized hyperbolicity on Banach spaces, should be replaced by something else.
If $S$ is an operator on a Banach space $Y$, it follows from the spectral radius formula that the condition $r(S) < 1$
is equivalent to the existence of constants $c > 0$ and $t \in (0,1)$ such that
\[
\|S^n y\| \leq c\, t^n \|y\| \ \ \text{ whenever } y \in Y \text{ and } n \in \N_0.
\]
On the other hand, there are some classical situations in which the spectral conditions in the concept of hyperbolicity are replaced
by the validity of certain exponential estimates.
This is the case in the definition of a {\em hyperbolic set} of a diffeomorphism in the area of differentiable dynamics 
\cite[Section~5.2]{MBriGStu02}.
As another example, we can mention the concept of {\em exponential dichotomy} in non-autonomous dynamics \cite[Definition~2.6]{KPal00},
which extends the idea of hyperbolicity to the non-autonomous case.
In view of these considerations, we propose the definition below.

\begin{definition}\label{gh-def}
Let $X$ be a locally convex space over $\K$ whose topology is induced by a directed family $(\|\cdot\|_\alpha)_{\alpha \in I}$ of seminorms.
We say that an operator $T \in L(X)$ is {\em generalized hyperbolic} if there is a topological direct sum decomposition
\begin{equation}\label{oplus}
X = M \oplus N
\end{equation}
with the following properties:
\begin{itemize}
\item [(GH1)] $T(M) \subset M$.
\item [(GH2)] $T(N) \supset N$ and $T|_N : N \to T(N)$ is an isomorphism.
\item [(GH3)] For every $\alpha \in I$, there exist $\beta \in I$, $c > 0$ and $t \in (0,1)$ such that
  \begin{equation}\label{GHI}
  \|T^n y\|_\alpha \leq c\, t^n \|y\|_\beta \ \text{ and } \ \|S^n z\|_\alpha \leq c\, t^n \|z\|_\beta \
  \text{ whenever } y \in M, z \in N, n \in \N_0,
  \end{equation}
  where $S\!:= (T|_N)^{-1}|_N \in L(N)$.
\end{itemize}
If $T \in GL(X)$, then condition (GH2) is equivalent to
\begin{itemize}
\item [(GH2')] $T^{-1}(N) \subset N$,
\end{itemize}
and the second inequality in (\ref{GHI}) can be rewritten as
\[
\|T^{-n} z\|_\alpha \leq c\, t^n \|z\|_\beta.
\]
If $M$ and $N$ are $T$-invariant, then we say that the operator $T$ is {\em hyperbolic}.
If, in addition, $M = \{0\}$ or $N = \{0\}$, then we say that the hyperbolic operator $T$ has {\em trivial splitting}.
\end{definition}

Note that the above notions are independent of the choice of the directed family $(\|\cdot\|_\alpha)_{\alpha \in I}$ of seminorms
inducing the topology of $X$. Moreover, they generalize to operators on arbitrary locally convex spaces the corresponding notions
for operators on Banach spaces.

As we mentioned in the Introduction, \cite[Theorem~A]{BerCirDarMesPuj18} asserts that every invertible generalized hyperbolic operator
on a Banach space has the shadowing property.
Surprisingly enough, there exist invertible hyperbolic operators on the Fr\'echet space $H(\C)$ of all entire functions that do not have the
shadowing property (Remark~\ref{NotShad}).
Nevertheless, we will prove below that we can always guarantee the finite shadowing property for generalized hyperbolic operators
on locally convex spaces.
Also, it was proved in \cite[Theorem~18]{NBerAPerTA} that generalized hyperbolic operators on Banach spaces have the strict periodic
shadowing property.
We will see below that the arguments in \cite{NBerAPerTA} can be adapted to extend this result to
arbitrary sequentially complete locally convex spaces, but we will also give a counterexample showing that the hypothesis
of sequential completeness is essential for the validity of this result (Remark~\ref{CounterexCompleteness}).
Recall that a topological vector space $X$ is said to be {\em sequentially complete} if every Cauchy sequence in $X$
converges to a point in $X$, where a sequence $(x_n)_{n \in \N}$ in $X$ is a {\em Cauchy sequence} if
for every neighborhood $V$ of $0$ in $X$, there exists $n_0 \in \N$ such that $x_n - x_m \in V$ whenever $n,m \geq n_0$.

Before stating and proving our results, let us recall the basic definitions related to the concept of shadowing.
If $T \in L(X)$ and $U$ is a neighborhood of $0$ in $X$, recall that a {\em $U$-pseudotrajectory} of $T$ is a finite or infinite sequence
$(x_j)_{i < j < k}$ in $X$ ($-\infty \leq i < k \leq \infty$), consisting of at least two terms, such that
\[
Tx_j - x_{j+1} \in U \ \ \text{ for all } i < j < k-1.
\]
A finite $U$-pseudotrajectory of the form $(x_j)_{j=0}^k$ is also called a {\em $U$-chain} for $T$ (from $x_0$ to $x_k$).
If, in addition, $x_k = x_0$, then we say that $(x_j)_{j=0}^k$ is a {\em $U$-cycle} for $T$.
Recall that $T$ has the {\em finite shadowing property} (resp.\ the {\em positive shadowing property}) if for every neighborhood $V$
of $0$ in $X$, there is a neighborhood $U$ of $0$ in $X$ such that every $U$-chain $(x_j)_{j=0}^k$ (resp.\ every $U$-pseudotrajectory
$(x_j)_{j \in \N_0}$) of $T$ is {\em $V$-shadowed} by the trajectory of some $x \in X$, in the sense that
\[
x_j - T^j x \in V \ \ \text{ for all } j \in \{0,\ldots,k\} \ \text{ (resp.\ for all } j \in \N_0).
\]
If $T \in GL(X)$, then the {\em shadowing property} is defined by replacing the set $\N_0$ by the set $\Z$ in the definition of
positive shadowing.
Recall also that $T$ has the {\em periodic shadowing property} \cite{PKos05,OsiPilTik10} if for every neighborhood $V$ of $0$ in $X$,
there is a neighborhood $U$ of $0$ in $X$ such that every periodic $U$-pseudotrajectory $(x_j)_{j \in \N_0}$ of $T$ is $V$-shadowed
by some $x \in \Per(T)$ (the set of all periodic points of $T$).
By adding the condition that the periodic point $x$ can be chosen to satisfy $T^kx = x$ if $(x_j)_{j \in \N_0}$ has period $k$,
then we obtain the {\em strict periodic shadowing property} \cite{NKaw2019}.

\begin{theorem}\label{GHFSP}
Every generalized hyperbolic operator $T$ on a locally convex space $X$ has the finite shadowing property.
If, in addition, $X$ is sequentially complete, then $T$ also has the strict periodic shadowing property.
\end{theorem}

\begin{proof}
Choose a directed family $(\|\cdot\|_\alpha)_{\alpha \in I}$ of seminorms inducing the topology of $X$
and let $M$, $N$ and $S$ be as in the definition of generalized hyperbolicity.
Let $P_M : X \to M$ and $P_N : X \to N$ be the canonical projections associated to the direct sum decomposition (\ref{oplus}).
Given a neighborhood $V$ of $0$ in $X$, there exist $\alpha \in I$ and $\eps > 0$ such that
\[
\{x \in X : \|x\|_\alpha < \eps\} \subset V.
\]
Let $\beta \in I$, $c > 0$ and $t \in (0,1)$ be such that (\ref{GHI}) holds.
Since the projections $P_M$ and $P_N$ are continuous, there exist $\gamma \in I$ and $d > 0$ such that
\begin{equation}\label{GHA}
\|P_M x\|_\beta \leq d\, \|x\|_\gamma \ \ \text{ and } \ \ \|P_N x\|_\beta \leq d\, \|x\|_\gamma \ \ \text{ for all } x \in X.
\end{equation}
Let $\delta\!:= \frac{(1-t)\eps}{3\,c\,d} > 0$ and $U\!:= \{x \in X : \|x\|_\gamma < \delta\}$, which is a neighborhood of $0$ in $X$.
Let $(x_j)_{j=0}^p$ be a $U$-chain for $T$ and define $y_j\!:= x_{j+1} - Tx_j$ for each $j \in \{0,\ldots,p-1\}$. Then
\begin{equation}\label{GHB}
\|y_j\|_\gamma < \delta \ \ \text{ for all } j \in \{0,\ldots,p-1\}.
\end{equation}
Consider the vector
\begin{equation}\label{GHD}
x\!:= x_0 + \sum_{j=1}^p S^j P_N y_{j-1}.
\end{equation}
We claim that
\begin{equation}\label{GHC}
x_m - T^m x = \sum_{j=0}^{m-1} T^j P_M y_{m-j-1} - \sum_{j=1}^{p-m} S^j P_N y_{m+j-1}
\ \ \text{ for all } m \in \{0,\ldots,p\},
\end{equation}
where sums of the form $\sum_{j=0}^{-1} a_j$ and $\sum_{j=1}^0 a_j$ are assumed to be $0$.
Clearly, (\ref{GHC}) holds for $m = 0$. Assume that it holds for a certain $m \in \{0,\ldots,p-1\}$. Then
\begin{align*}
x_{m+1} - T^{m+1} x
  &= y_m + T(x_m - T^m x)\\
  &= P_My_m + P_Ny_m + \sum_{j=0}^{m-1} T^{j+1} P_M y_{m-j-1} - \sum_{j=1}^{p-m} S^{j-1} P_N y_{m+j-1}\\
  &= \sum_{j=0}^{m} T^j P_M y_{m-j} - \sum_{j=1}^{p-m-1} S^j P_N y_{m+j}.
\end{align*}
Hence, by induction, we obtain (\ref{GHC}). By (\ref{GHI}), (\ref{GHA}), (\ref{GHB}) and (\ref{GHC}), for each $m \in \{0,\ldots,p\}$,
\begin{align*}
\|x_m - T^m x\|_\alpha
  &\leq \sum_{j=0}^{m-1} \|T^j P_M y_{m-j-1}\|_\alpha + \sum_{j=1}^{p-m} \|S^j P_N y_{m+j-1}\|_\alpha\\
  &\leq \sum_{j=0}^{m-1} c\,t^j \|P_M y_{m-j-1}\|_\beta + \sum_{j=1}^{p-m} c\,t^j \|P_N y_{m+j-1}\|_\beta\\
  &\leq \sum_{j=0}^{m-1} c\,d\,t^j \|y_{m-j-1}\|_\gamma + \sum_{j=1}^{p-m} c\,d\,t^j \|y_{m+j-1}\|_\gamma\\
  &< \frac{2\,c\,d\,\delta}{1-t} < \eps.
\end{align*}
This shows that the $U$-chain $(x_j)_{j=0}^p$ is $V$-shadowed by $(T^jx)_{j=0}^p$, proving that $T$ has the finite shadowing property.

Now, assume that $X$ is sequentially complete.
Assume also that $(x_j)_{j=0}^p$ is a $U$-cycle for $T$.
We extend this $U$-cycle to a periodic $U$-pseudotrajectory of $T$ by defining
\[
(x_j)_{j \in \N_0}\!:= (x_0,x_1,\ldots,x_p,x_1,\ldots,x_p,x_1,\ldots,x_p,\ldots).
\]
Let $y_j\!:= x_{j+1} - Tx_j$ for each $j \in \N_0$.
Note that the sequence $(y_j)_{j \in \N_0}$ is also periodic with period $p$ and that
\[
\|y_j\|_\gamma < \delta \ \ \text{ for all } j \in \N_0.
\]
The vector $x$ defined in (\ref{GHD}) may fail to be periodic.
However, we can follow the proof of \cite[Theorem~18]{NBerAPerTA} and replace (\ref{GHD}) by the following definition:
\begin{equation}\label{GHE}
x\!:= x_0 + \sum_{j=1}^\infty S^j P_N y_{j-1} - \sum_{j=0}^{p-1} \sum_{k=0}^\infty T^{kp+j} P_M y_{p-j-1}.
\end{equation}
Since the sequence $(y_j)_{j \in \N_0}$ is bounded, the estimates in (\ref{GHI}) imply that the sequence of partial sums
of each infinite series in (\ref{GHE}) is a Cauchy sequence.
Since we are assuming that $X$ is sequentially complete, the vector $x$ is well-defined.
The computations done in the proof of \cite[Theorem~18]{NBerAPerTA} show that
\[
T^p x = x
\]
and
\[
x_m - T^m x = \sum_{j=0}^{m-1} T^j P_M y_{m-j-1} - \sum_{j=1}^\infty S^j P_N y_{m+j-1}
   + \sum_{j=0}^{p-1} \sum_{k=0}^\infty T^{kp+j+m} P_M y_{p-j-1},
\]
for all $m \in \{0,\ldots,p\}$. Hence, $x$ is a periodic point of $T$ with period $p$ and estimates like to the ones we made before give
\[
\|x_m - T^mx\|_\alpha < \frac{3\,c\,d\,\delta}{1-t} = \eps \ \ \text{ for all } m \in \{0,\ldots,p\}.
\]
This proves that $T$ has the strict periodic shadowing property.
\end{proof}

Let $X$ be a topological vector space. Given $T \in L(X)$, recall that $x \in X$ is a {\em chain recurrent point} of $T$
if for every neighborhood $U$ of $0$ in $X$, there is a $U$-chain (actually a $U$-cycle) for $T$ from $x$ to itself.
The set $\CR(T)$ of all chain recurrent points of $T$ is called the {\em chain recurrent set} of $T$ and
$T$ is {\em chain recurrent} if $\CR(T) = X$. Clearly, $\Per(T) \subset \CR(T)$.
Recall also that $T$ is {\em topologically transitive} (resp.\ {\em topologically mixing}) if for any pair $A,B$ of nonempty open sets in $X$,
there exists $n \in \N_0$ (resp.\ $n_0 \in \N_0$) such that $T^n(A) \cap B \neq \emptyset$ (resp.\ for all $n \geq n_0$).
Finally, recall that $T$ is {\em Devaney chaotic} if $T$ is transitive and $\Per(T)$ is dense in $X$.

Operators with the periodic shadowing property have some additional interesting dynamical properties,
which are described in the next theorem.

\begin{theorem}\label{PSPEVT}
Let $X$ be a topological vector space.
If $T \in L(X)$ has the (strict) periodic shadowing property, then the following properties hold:
\begin{itemize}
\item [\rm (a)] $\Per(T)$ is dense in $\CR(T)$.
\item [\rm (b)] $T|_{\CR(T)}$ has the finite shadowing property and the (strict) periodic shadowing property.
\item [\rm (c)] $T|_{\CR(T)}$ is topologically mixing and Devaney chaotic.
\end{itemize}
\end{theorem}

\begin{proof}
By \cite[Proposition~26]{NBerAPerTA}, $Y\!:= \CR(T)$ is a closed $T$-invariant subspace of $X$.
Hence, $S\!:= T|_Y$ is a well-defined continuous linear operator on $Y$.
Given a symmetric neighborhood $V$ of $0$ in $X$, let $U$ be a neighborhood of $0$ in $X$ associated to $V$
according to the fact that $T$ has the (strict) periodic shadowing property.

\smallskip\noindent
\textbf{ Step 1.} $\Per(T)$ is dense in $\CR(T)$:

\smallskip
Given $x \in \CR(T)$, there exists a $U$-chain for $T$ from $x$ to itself.
By periodic shadowing, this $U$-cycle for $T$ must be $V$-shadowed by the trajectory under $T$ of some $y \in \Per(T)$.
In particular, $y \in x+V$.

\smallskip\noindent
\textbf{ Step 2.} $S$ has the finite shadowing property:

\smallskip
Let $(x_j)_{j=0}^k$ be a $(U \cap Y)$-chain for $S$. Since $x_0,x_k \in \CR(T)$, \cite[Proposition~25]{NBerAPerTA} says
that there is a $U$-chain $(x_k,x_{k+1},\ldots,x_p)$ for $T$ from $x_k$ to $x_0$.
Hence, $(x_j)_{j=0}^p$ is a $U$-cycle for $T$.
By periodic shadowing, it must be $V$-shadowed by the trajectory under $T$ of some $x \in \Per(T)$.
In particular,
\[
x \in Y \ \ \text{ and } \ \ x_j - S^jx \in V \cap Y \text{ for all } j \in \{0,\ldots,k\}.
\]

\smallskip\noindent
\textbf{ Step 3.} $S$ has the (strict) periodic shadowing property:

\smallskip
If $(x_j)_{j=0}^k$ is a $(U \cap Y)$-cycle for $S$, then it is also a $U$-cycle for $T$, and so it must be $V$-shadowed
by the trajectory under $T$ of some $x \in \Per(T)$.
Then, $x \in Y$ and $(x_j)_{j=0}^k$ is $(V \cap Y)$-shadowed by the trajectory under $S$ of $x$.
In the strict case, $x$ can be chosen so that $T^kx = x$, that is, $S^kx = x$.

\smallskip\noindent
\textbf{ Step 4.} $S$ is topologically mixing:

\smallskip
By Step~1,
\begin{equation}\label{PSPEVT1}
\CR(S) \subset \CR(T) = \ov{\Per(T)^{\,X}} = \ov{\Per(S)^{\,X}} = \ov{\Per(S)^{\,Y}} \subset \CR(S).
\end{equation}
Hence, $\CR(S) = \CR(T) = Y$, proving that $S$ is a chain recurrent operator.
Since the concepts of chain recurrence and topological mixing coincide for operators with the finite shadowing property
\cite[Theorem~7]{NBerAPerTA} (see also \cite[Theorem~A]{AntManVar22}), we conclude from Step~2 that $S$ is topologically mixing.

\smallskip\noindent
\textbf{ Step 5.} $S$ is Devaney chaotic:

\smallskip
By Step~4, $S$ is topologically transitive, and by (\ref{PSPEVT1}), $\Per(S)$ is dense in $Y$.
\end{proof}

By combining Theorems~\ref{GHFSP} and~\ref{PSPEVT}, we obtain the following result.

\begin{corollary}\label{GHFSPCor}
Let $X$ be a sequentially complete locally convex space.
If $T \in L(X)$ is generalized hyperbolic, then the following properties hold:
\begin{itemize}
\item [\rm (a)] $\Per(T)$ is dense in $\CR(T)$.
\item [\rm (b)] $T|_{\CR(T)}$ has the finite shadowing property and the strict periodic shadowing property.
\item [\rm (c)] $T|_{\CR(T)}$ is topologically mixing and Devaney chaotic.
\end{itemize}
\end{corollary}

For invertible generalized hyperbolic operators on Banach spaces, it was proved in \cite{CirGolPuj21} that
$T|_{\Omega(T)}$ is topologically mixing and Devaney chaotic, where $\Omega(T)$ denotes the set of all nonwandering points of $T$.
Recall that $x \in X$ is a {\em nonwandering point} of $T$ if for every neighborhood $A$ of $x$ in $X$,
there exists $n \in \N$ such that $T^n(A) \cap A \neq \emptyset$.
This result from \cite{CirGolPuj21} can be seen as a particular case of the above corollary,
since the denseness of $\Per(T)$ in $\CR(T)$ implies that $\Omega(T) = \CR(T)$.

\begin{remark}\label{CounterexCompleteness}
The hypothesis of sequential completeness is essential for the validity of the second assertion in Theorem~\ref{GHFSP}
and for the validity of Corollary~\ref{GHFSPCor}, even in the case of normed spaces.
As a counterexample, let $X$ be the vector space of all sequences $(x_n)_{n \in \Z}$ of scalars with finite support
endowed with any $\ell_p$-norm ($1 \leq p \leq \infty$) and let $T \in GL(X)$ be the generalized hyperbolic operator given by
\[
T((x_n)_{n \in \Z})\!:= (w_{n+1} x_{n+1})_{n \in \Z},
\]
where $w_n\!:= 1/2$ if $n \leq 0$ and $w_n\!:= 2$ if $n \geq 1$.
Given any $\delta > 0$, choose $n \in \N$ such that $2^n\delta > 1$. Then,
\[
(0,\delta e_n, 2\delta e_{n-1},\ldots, 2^{n-1} \delta e_1, 2^n \delta e_0, 2^{n-1} \delta e_{-1},\ldots, 2 \delta e_{-n+1}, \delta e_{-n},0)
\]
is a $\delta$-cycle for $T$ that cannot be $1$-shadowed by a periodic point of $T$, since $\Per(T) = \{0\}$.
Moreover, $\CR(T) = X$, that is, $T$ is chain recurrent, but $T$ is not Devaney chaotic.
\end{remark}

Our next goal is to present additional conditions to guarantee the validity of the shadowing property
for generalized hyperbolic operators on locally convex spaces.

Given a seminorm $\|\cdot\|$ on a vector space $X$, we define the {\em kernel} of $\|\cdot\|$ by
\[
\Ker(\|\cdot\|)\!:= \{x \in X : \|x\| = 0\}.
\]
A sequence $(x_n)_{n \in \N}$ in $X$ is said to be a {\em Cauchy sequence with respect to} $\|\cdot\|$ if
$\|x_n - x_m\| \to 0$ as $n,m \to \infty$.
The seminorm $\|\cdot\|$ is said to be {\em complete} if every Cauchy sequence $(x_n)_{n \in \N}$ with respect to $\|\cdot\|$
has a {\em limit} $x \in X$, in the sense that $\|x_n - x\| \to 0$ as $n \to \infty$.
In this case, note that the set of all limits of the sequence $(x_n)_{n \in \N}$ is exactly $x + \Ker(\|\cdot\|)$.

\begin{theorem}\label{GHSP}
Suppose that the topology of a locally convex space $X$ is induced by a directed family $(\|\cdot\|_\alpha)_{\alpha \in I}$
of {\em complete} seminorms.
\begin{itemize}
\item [\rm (a)] If $T \in GL(X)$ is generalized hyperbolic and $T(\Ker(\|\cdot\|_\alpha)) = \Ker(\|\cdot\|_\alpha)$ for all $\alpha \in I$,
  then $T$ has the shadowing property.
\item [\rm (b)] If $T \in L(X)$ is generalized hyperbolic and $T(\Ker(\|\cdot\|_\alpha)) \subset \Ker(\|\cdot\|_\alpha)$ for all $\alpha \in I$,
  then $T$ has the positive shadowing property.
\end{itemize}
\end{theorem}

\begin{proof}
(a): Let $X = M \oplus N$ be the topological direct sum decomposition given by the genera\-lized hyperbolicity of $T$
and let $P_M : X \to M$ and $P_N : X \to N$ be the canonical projections.
Given a neighborhood $V$ of $0$ in $X$, there exist $\theta \in I$ and $\eps > 0$ such that
\[
\{x \in X : \|x\|_\theta < \eps\} \subset V.
\]
Choose $\alpha \in I$ and $a > 0$ such that
\begin{equation}\label{GHSP0}
\|x\|_\theta \leq \|x\|_\alpha, \ \ \|Tx\|_\theta \leq a \|x\|_\alpha \ \text{ and } \ \|T^{-1}x\|_\theta \leq a \|x\|_\alpha
\ \text{ for all } x \in X.
\end{equation}
Let $\beta \in I$, $c > 0$ and $t \in (0,1)$ be such that
\[
  \|T^n y\|_\alpha \leq c\, t^n \|y\|_\beta \ \text{ and } \ \|T^{-n} z\|_\alpha \leq c\, t^n \|z\|_\beta \
  \text{ for all } y \in M, z \in N, n \in \N_0.
\]
Choose $\gamma \in I$ and $d > 0$ such that
\[
\|P_M x\|_\beta \leq d\, \|x\|_\gamma \ \text{ and } \ \|P_N x\|_\beta \leq d\, \|x\|_\gamma \ \text{ for all } x \in X.
\]
Let $\eta \in I$ and $b \geq 1$ satisfy
\[
\|x\|_\gamma \leq \|x\|_\eta \ \text{ and } \ \|T^{-1}x\|_\gamma \leq b\, \|x\|_\eta \ \text{ for all } x \in X.
\]
Put $\delta\!:= \frac{(1-t)\eps}{3\,b\,c\,d} > 0$ and $U\!:= \{x \in X : \|x\|_\eta < \delta\}$.
Let $(x_j)_{j \in \Z}$ be a $U$-pseudotrajectory of $T$ and define
$y_j\!:= x_{j+1} - Tx_j$ and $z_j\!:= x_{-j-1} - T^{-1}x_{-j}$ for each $j \in \N_0$. Then
\[
\|y_j\|_\gamma < \delta \ \text{ and } \ \|z_j\|_\gamma < b\,\delta \ \text{ for all } j \in \N_0.
\]
For each $k \in \N$, let
\[
u_k\!:= \sum_{j=1}^k T^{-j} P_N y_{j-1} \ \ \text{ and } \ \ v_k\!:= \sum_{j=1}^k T^j P_M z_{j-1}.
\]
For $m \in \Z$ and $j \in \N$, we have that
\[
\|T^{m-j} P_N y_{j-1}\|_\alpha \leq c\,t^{j-m} \|P_N y_{j-1}\|_\beta \leq c\,d\,t^{j-m} \|y_{j-1}\|_\gamma < c\,d\,\delta\, t^{j-m}
  \ \text{ if } j \geq m,
\]
and
\[
\|T^{m+j} P_M z_{j-1}\|_\alpha \leq c\,t^{j+m} \|P_M z_{j-1}\|_\beta \leq c\,d\,t^{j+m} \|z_{j-1}\|_\gamma < b\,c\,d\,\delta\, t^{j+m}
  \ \text{ if } j \geq -m.
\]
This implies that the sequences $(T^m u_k)_{k \in \N}$ and $(T^m v_k)_{k \in \N}$ are Cauchy with respect to $\|\cdot\|_\alpha$.
Since we are assuming that $\|\cdot\|_\alpha$ is complete, there exist $p_m,q_m \in X$ such that
\[
\|p_m - T^m u_k\|_\alpha \to 0 \ \text{ and } \ \|q_m - T^m v_k\|_\alpha \to 0 \ \text{ as } k \to \infty \ \ \ \ (m \in \Z).
\]

Let us prove that
\begin{equation}\label{GHSPA}
\|T^m p_0 - T^m u_k\|_\theta \to 0 \ \text{ as } k \to \infty, \ \text{ for all } m \in \Z.
\end{equation}

The case $m = 0$ is clear. Suppose that (\ref{GHSPA}) holds for a certain $m \in \Z$. Since
\[
\|p_{m-1} - T^{m-1} u_k\|_\alpha \to 0 \ \text{ and } \ \|p_{m+1} - T^{m+1} u_k\|_\alpha \to 0 \ \text{ as } k \to \infty,
\]
we obtain from (\ref{GHSP0}) that
\[
\|Tp_{m-1} - T^m u_k\|_\theta \to 0 \ \text{ and } \ \|T^{-1}p_{m+1} - T^m u_k\|_\theta \to 0 \ \text{ as } k \to \infty.
\]
The induction hypothesis tell us that $\|T^m p_0 - T^m u_k\|_\theta \to 0$ as $k \to \infty$. Therefore,
\[
T p_{m-1} - T^m p_0 \in \Ker(\|\cdot\|_\theta) \ \ \text{ and } \ \ T^{-1} p_{m+1} - T^m p_0 \in \Ker(\|\cdot\|_\theta).
\]
Since $T(\Ker(\|\cdot\|_\theta)) = \Ker(\|\cdot\|_\theta)$, we conclude that
\[
p_{m-1} - T^{m-1} p_0 \in \Ker(\|\cdot\|_\theta) \ \ \text{ and } \ \ p_{m+1} - T^{m+1} p_0 \in \Ker(\|\cdot\|_\theta),
\]
which implies that (\ref{GHSPA}) holds with $m-1$ and with $m+1$ in the place of $m$.

Analogously,
\begin{equation}\label{GHSPB}
\|T^m q_0 - T^m v_k\|_\theta \to 0 \ \text{ as } k \to \infty, \ \text{ for all } m \in \Z.
\end{equation}

Now, by arguing as in the proof of (\ref{GHC}), we obtain
\[
x_m - T^m(x_0 + u_k) = \sum_{j=0}^{m-1} T^j P_M y_{m-j-1} - \sum_{j=1}^{k-m} T^{-j} P_N y_{m+j-1}
\]
and
\[
x_{-m} - T^{-m}(x_0 + v_k) = \sum_{j=0}^{m-1} T^{-j} P_N z_{m-j-1} - \sum_{j=1}^{k-m} T^j P_M z_{m+j-1},
\]
whenever $0 \leq m \leq k$. Hence,
\begin{align}\label{GHSPC}
\|x_m - T^m(x_0 + u_k + v_k)\|_\alpha
  &\leq \|x_m - T^m(x_0 + u_k)\|_\alpha + \|T^m v_k\|_\alpha\\
  &\leq c\,d\,\delta \Bigg(\sum_{j=0}^{m-1} t^j + \sum_{j=1}^{k-m} t^j\Bigg) + b\,c\,d\,\delta \sum_{j=1}^{k} t^{m+j} \notag
\end{align}
and
\begin{align}\label{GHSPD}
\|x_{-m} - T^{-m}(x_0 + u_k + v_k)\|_\alpha
  &\leq \|x_{-m} - T^{-m}(x_0 + v_k)\|_\alpha + \|T^{-m} u_k\|_\alpha\\
  &\leq b\,c\,d\,\delta \Bigg(\sum_{j=0}^{m-1} t^j + \sum_{j=1}^{k-m} t^j\Bigg) + c\,d\,\delta \sum_{j=1}^{k} t^{m+j}, \notag
\end{align}
whenever $0 \leq m \leq k$. Consider the vector
\[
x\!:= x_0 + p_0 + q_0.
\]
By fixing $m \in \N_0$ and letting $k \to \infty$, it follows from (\ref{GHSPA}), (\ref{GHSPB}), (\ref{GHSPC}) and (\ref{GHSPD}) that
\[
\|x_m - T^m x\|_\theta \leq b\,c\,d\,\delta \Bigg(\sum_{j=0}^{m-1} t^j + 2\sum_{j=1}^{\infty} t^j\Bigg) < \eps
\]
and
\[
\|x_{-m} - T^{-m} x\|_\theta \leq b\,c\,d\,\delta \Bigg(\sum_{j=0}^{m-1} t^j + 2\sum_{j=1}^{\infty} t^j\Bigg) < \eps.
\]
Thus, the $U$-pseudotrajecotry $(x_j)_{j \in \Z}$ of $T$ is $V$-shadowed by the trajectory of $x$ under $T$,
proving that $T$ has the shadowing property.

\smallskip\noindent
(b): The proof is analogous (but slightly simpler) and is left to the reader.
\end{proof}

\begin{remark}
As an immediate consequence of Theorem~\ref{GHSP}, we obtain the following known results:
\begin{itemize}
\item Every invertible generalized hyperbolic operator on a Banach space has the shadowing property \cite[Theorem~A]{BerCirDarMesPuj18}.
\item Every generalized hyperbolic operator on a Banach space has the positive shadowing property \cite[Proposition~19]{NBerAMesArxiv}.
\end{itemize}
These results can also be derived directly from Theorem~\ref{GHFSP}, since it was proved in \cite[Theorem~1]{NBerAPerTA} that
(positive) shadowing and finite shadowing coincide for operators on Banach spaces.
\end{remark}

\begin{example}\label{ExampleH}
Let $H(\C)$ be the Fr\'echet space of all entire functions endowed with the compact-open topology,
that is, the topology induced by the sequence of seminorms (actually norms) given by
\begin{equation}\label{SN-1}
\|f\|_k\!:= \max_{|z| \leq k} |f(z)| \ \ \ (f \in H(\C), \ k \in \N).
\end{equation}
Given a zero-free entire function $\phi$, consider the multiplication operator
\[
M_\phi : f \in H(\C) \mapsto \phi \cdot f \in H(\C).
\]
We have that the following assertions are equivalent:
\begin{itemize}
\item [(i)] $M_\phi$ has the finite shadowing property;
\item [(ii)] $M_\phi$ is generalized hyperbolic;
\item [(iii)] $M_\phi$ is hyperbolic with trivial splitting;
\item [(iv)] $\phi$ is a constant function with modulus $\neq 1$.
\end{itemize}
Indeed, (iv) $\Rightarrow$ (iii) is easy, (iii) $\Rightarrow$ (ii) is obvious, and (ii) $\Rightarrow$ (i) follows from Theorem~\ref{GHFSP}.
It remains to prove that (i) $\Rightarrow$ (iv). For this purpose, suppose that (i) is true and (iv) is false.
Then, there exists $z_0 \in \C$ with $|\phi(z_0)| = 1$ and there exist $H \subset \C$ compact and $\delta > 0$ such that
the condition
\begin{equation}\label{ExHC1}
|(M_\phi(f_j))(z) - f_{j+1}(z)| \leq \delta \ \ \text{ for all } z \in H \text{ and } j \in \{0,\ldots,k-1\},
\end{equation}
where $k \in \N$ and $f_0,\ldots,f_k \in H(\C)$, implies the existence of an $f \in H(\C)$ with
\begin{equation}\label{ExHC2}
|f_j(z_0) - ((M_\phi)^j(f))(z_0)| < 1 \ \ \text{ for all } j \in \{0,\ldots,k\}.
\end{equation}
Note that (\ref{ExHC2}) implies that
\begin{equation}\label{ExHC3}
|f_k(z_0)| < 2 + |f_0(z_0)|.
\end{equation}
However, by defining $f_0\!:= 0$ and $f_j\!:= \phi \cdot f_{j-1} + \phi(z_0)^{j-1} \cdot \delta$ for $j \geq 1$, we have that
(\ref{ExHC1}) holds for every $k \in \N$, but (\ref{ExHC3}) fails for $k$ large enough. This contradiction completes the proof.
\end{example}

\begin{remark}\label{NotShad}
We have just seen that the invertible multiplication operators on $H(\C)$ that have the finite shadowing property are exactly those of the form
\[
M_\lambda : f \in H(\C) \mapsto \lambda f \in H(\C),
\]
where $\lambda \in \C$ and $|\lambda| \not\in \{0,1\}$.
Nevertheless, it was proved in \cite[Theorem~2]{NBerAPerTA} that none of these operators has the shadowing property.
More precisely, $M_\lambda$ has the finite shadowing property but does not have the positive shadowing property whenever
$|\lambda| > 1$, and $M_\lambda$ has the positive shadowing property but does not have the shadowing property whenever
$0 < |\lambda| < 1$ (see \cite[Remark~3]{NBerAPerTA}).
This shows that, in general, an invertible hyperbolic operator may fail to have the (positive) shadowing property,
even in the Fr\'echet space setting.
This also shows that the hypothesis of completeness of the seminorms $\|\cdot\|_\alpha$ is essential for the validity of Theorem~\ref{GHSP}
(note that the seminorms in (\ref{SN-1}) are not complete).
\end{remark}

In strong contrast to the case of the space $H(\C)$ of entire functions, we will now see that shadowing and finite shadowing
coincide for invertible multiplication operators on $C(\Omega)$ spaces.

\begin{example}\label{ExampleC}
Let $\Omega$ be a locally compact Hausdorff space and let $C(\Omega)$ be the complete locally convex space of all continuous maps
$f : \Omega \to \K$ endowed with the compact-open topology, that is, the topology induced by the family of seminorms
\begin{equation}\label{SN-2}
\|f\|_K\!:= \max_{x \in K} |f(x)| \ \ \ (f \in C(\Omega), \ K \subset \Omega \text{ nonempty and compact}).
\end{equation}
Note that $C(\Omega)$ is a Fr\'echet space if $\Omega$ is $\sigma$-compact.
Given a zero-free continuous map $\phi : \Omega \to \K$, consider the multiplication operator
\[
M_\phi : f \in C(\Omega) \mapsto \phi \cdot f \in C(\Omega).
\]
We have that the following assertions are equivalent:
\begin{itemize}
\item [(i)] $M_\phi$ has the finite shadowing property;
\item [(ii)] $M_\phi$ has the shadowing property;
\item [(iii)] $M_\phi$ is generalized hyperbolic;
\item [(iv)] $M_\phi$ is hyperbolic;
\item [(v)] $|\phi(x)| \neq 1$ for every $x \in \Omega$.
\end{itemize}
Indeed, contrary to the case of Example~\ref{ExampleH}, the seminorms in (\ref{SN-2}) are {\em complete}.
Thus, (iii) $\Rightarrow$ (ii) follows from Theorem~\ref{GHSP}.
Moreover, (ii) $\Rightarrow$ (i) is obvious, (i) $\Rightarrow$ (v) is analogous to the proof of (i) $\Rightarrow$ (iv) in the previous example,
and (iv) $\Rightarrow$ (iii) is also obvious.
It remains to prove that (v) $\Rightarrow$ (iv). For this purpose, assume (v) and define
\[
A\!:= \{x \in \Omega : |\phi(x)| < 1\} \ \ \text{ and } \ \ \ B\!:= \{x \in \Omega : |\phi(x)| > 1\}.
\]
In view of (v), we have that $\Omega = A \cup B$ and both $A$ and $B$ are simultaneously open and closed in $\Omega$. Let
\[
M\!:= \{f \in C(\Omega) : f = 0 \text{ on } B\} \ \ \text{ and } \ \ N\!:= \{f \in C(\Omega) : f = 0 \text{ on } A\},
\]
which are closed $M_\phi$-invariant subspaces of $C(\Omega)$.
It is easy to check that $C(\Omega) = M \oplus N$ as a topological direct sum.
In order to prove (GH3), given a nonempty compact subset $K$ of $\Omega$, define
\[
t\!:= \max\Big\{\max_{x \in K \cap A} |\phi(x)|, \max_{x \in K \cap B} \frac{1}{|\phi(x)|}\Big\} < 1,
\]
and note that
\[
\|(M_\phi)^nf\|_K \leq t^n \|f\|_K \ \text{ and } \ \|(M_\phi)^{-n}g\|_K \leq t^n \|g\|_K,
\]
whenever $f \in M$, $g \in N$ and $n \in \N_0$. Hence, (iv) holds.
\end{example}

Example~\ref{ExampleH} can be generalized as follows.

\begin{example}
Let $H(\Omega)$ be the Fr\'echet space of all holomorphic functions $f : \Omega \to \C$ endowed with the compact-open topology,
where $\Omega$ is a nonempty open set in $\C$.
Given a zero-free holomorphic function $\phi : \Omega \to \C$, consider the multiplication operator
\[
M_\phi : f \in H(\Omega) \mapsto \phi \cdot f \in H(\Omega).
\]
We have that the following assertions are equivalent:
\begin{itemize}
\item [(i)] $M_\phi$ has the finite shadowing property;
\item [(ii)] $M_\phi$ is generalized hyperbolic;
\item [(iii)] $M_\phi$ is hyperbolic;
\item [(iv)] $|\phi(z)| \neq 1$ for all $z \in \Omega$.
\end{itemize}
The proof of (i) $\Rightarrow$ (iv) is essentially the same as the proof given in Example~\ref{ExampleH},
and (iv)~$\Rightarrow$~(iii) follows from the argument used in Example~\ref{ExampleC}.
Contrary to the case of the space of entire functions, we cannot guarantee in general that the hyperbolic splitting in (iii) is trivial
nor that $\phi$ is a constant function in (iv).
\end{example}

It follows immediately from the definition that the inverse $T^{-1}$, the rotations $\lambda T$, $|\lambda| = 1$, and the powers
$T^n$, $n \in \N$, of a (generalized) hyperbolic operator $T \in GL(X)$ are (generalized) hyperbolic.

\begin{proposition}
Suppose that $X$ is the product of a family $(X_\ell)_{\ell \in L}$ of locally convex spaces, $T_\ell \in GL(X_\ell)$ for each $\ell \in L$,
and $T \in GL(X)$ is the product operator given by
\[
T((x_\ell)_{\ell \in L})\!:= (T_\ell x_\ell)_{\ell \in L}.
\]
If each $T_\ell$ is (generalized) hyperbolic, then so is $T$.
\end{proposition}

\begin{proof}
For each $\ell \in L$, let $(\|\cdot\|_{\alpha_\ell})_{\alpha_\ell \in I_\ell}$ be a directed family of seminorms
that induces the topology of $X_\ell$.
Recall that the product topology on $X$ is induced by the family of seminorms given by
\[
\|(x_\ell)_{\ell \in L}\|_{\alpha_{\ell_1},\ldots,\alpha_{\ell_k}}\!:=
  \max\{\|x_{\ell_1}\|_{\alpha_{\ell_1}},\ldots,\|x_{\ell_k}\|_{\alpha_{\ell_k}}\},
\]
for $k \in \N$, $\ell_1,\ldots,\ell_k \in L$ and $\alpha_{\ell_1} \in I_{\ell_1},\ldots,\alpha_{\ell_k} \in I_{\ell_k}$.

Suppose that each $T_\ell$ is generalized hyperbolic and let $X_\ell = M_\ell \oplus N_\ell$ be the topological direct sum decomposition
given by the definition of generalized hyperbolicity.
Consider the subspaces $M\!:= \prod_{\ell \in L} M_\ell$ and $N\!:= \prod_{\ell \in L} N_\ell$ of $X$.
It is routine to verify that $X = M \oplus N$ as a topological direct sum decomposition and that properties
(GH1), (GH2') and (GH3) hold.

If each $T_\ell$ is hyperbolic, then it is clear that $T$ is hyperbolic.
\end{proof}

Concrete examples of generalized hyperbolic operators on Banach (and Hilbert) spaces that are not hyperbolic can be found in \cite{BerCirDarMesPuj18,NBerAMes21}.
In Section~\ref{TEWS} we will exhibit some examples of this type on the (non-normable) Fr\'echet space $s(\Z)$ of rapidly decreasing sequences on~$\Z$.


\section{Topological stability for operators}\label{TSSection}

Our main goal in the present section is to establish the following result.

\begin{theorem}\label{TSO-FSP}
Let $X$ be a locally convex space.
If $T \in GL(X)$ is topologically stable, then $T$ has the finite shadowing property and the strict periodic shadowing property.
\end{theorem}

The operator $T$ is said to be {\em topologically stable} if for every neighborhood $V$ of $0$ in $X$,
there is a neighborhood $U$ of $0$ in $X$ such that for any homeomorphism $S : X \to X$ with
\[
Tx - Sx \in U \text{ for all } x \in X,
\]
there is a continuous map $\phi : X \to X$ satisfying
\[
T \circ \phi = \phi \circ S \ \ \ \text{ and } \ \ \ \phi(x) - x \in V \text{ for all } x \in X.
\]
We will talk more about this concept later in this section, where we will consider it in the more general setting of uniform spaces
(see Definition~\ref{TS-US} and the comments following it).

In view of Theorem~\ref{PSPEVT}, we obtain the following result.

\begin{corollary}
Let $X$ be a locally convex space.
If $T \in GL(X)$ is topologically stable, then the following properties hold:
\begin{itemize}
\item [\rm (a)] $\Per(T)$ is dense in $\CR(T)$.
\item [\rm (b)] $T|_{\CR(T)}$ has the finite shadowing property and the strict periodic shadowing property.
\item [\rm (c)] $T|_{\CR(T)}$ is topologically mixing and Devaney chaotic.
\end{itemize}
\end{corollary}

Since the notions of shadowing and finite shadowing coincide for operators on Banach spaces \cite[Theorem~1]{NBerAPerTA},
the above theorem implies the following result.

\begin{corollary}\label{TSO-FSP-Cor}
Let $X$ be a Banach space.
If $T \in GL(X)$ is topologically stable, then $T$ has the shadowing property and the strict periodic shadowing property.
\end{corollary}

In order to prove Theorem~\ref{TSO-FSP}, we shall introduce a certain concept of {\em multihomogeneity} for uniform spaces.
So, let us begin by fixing some notations regarding uniform spaces (see \cite[Chapter~II]{NBou1989} for the basics on uniform spaces).
Consider a uniform space $X$ with uniformity $\cU$.
Given $A \subset X$ and $U \in \cU$, the set
\[
U(A)\!:= \{y \in X : (x,y) \in U \text{ for some } x \in A\}
\]
is called the {\em $U$-neighborhood} of $A$ in $X$. If $A$ is a singleton, say $A = \{x\}$, it is usual to write $U(x)$ instead of $U(\{x\})$.
Recall that, for each $x \in X$, the family $\{U(x) : U \in \cU\}$ constitutes a fundamental system of neighborhoods of $x$
in the topology induced by the uniformity $\cU$.
Given $U,V \in \cU$, we define
\[
U \circ V\!:= \{(x,y) \in X \times X : (x,z) \in V \text{ and } (z,y) \in U \text{ for some } z \in X\}.
\]
Finally, $U \in \cU$ is said to be {\em symmetric} if $(y,x) \in U$ whenever $(x,y) \in U$.

Now, recall that a topological space $Y$ is said to be {\em homogeneous} if for any $a,b \in Y$,
there is a homeomorphism $h : Y \to Y$ with $h(a) = b$. This concept was introduced by Sierpi\'nski \cite{WSie1920} in 1920.
Since then, many variations, including several notions of local homogeneity,
have been introduced and investigated by several authors.
The definition given below can be seen as a ``multihomogeneous'' version of Ford's concept \cite{LFor1954} of a
{\em strongly locally homogeneous space}.

\begin{definition}
A uniform space $X$ with uniformity $\cU$ is said to be {\em strongly locally multihomogeneous} if for every $V \in \cU$,
there exists $U \in \cU$ such that $U \subset V$ and for any integer $k \geq 1$, any pairwise distinct points $a_1,\ldots,a_k \in X$ and
any pairwise distinct points $b_1,\ldots,b_k \in X$ with $(a_j,b_j) \in U$ for all $j \in \{1,\ldots,k\}$,
there is a homeomorphism $h : X \to X$ satisfying the following conditions:
\begin{itemize}
\item [\rm (a)] $h(a_j) = b_j$ for all $j \in \{1,\ldots,k\}$;
\item [\rm (b)] $h(U(a_j)) \subset V(a_j)$ for all $j \in \{1,\ldots,k\}$;
\item [\rm (c)] $h(x) = x$ for all $x \in X \backslash (U(a_1) \cup \ldots \cup U(a_k))$.
\end{itemize}
\end{definition}

\begin{remark}
If $X$ is strongly locally multihomogeneous, then the following property holds:

\smallskip\noindent
(P) For each $V \in \cU$, there exists $U \in \cU$ such that for any integer $k \geq 1$, any pairwise distinct points
$a_1,\ldots,a_k \in X$ and any pairwise distinct points $b_1,\ldots,b_k \in X$ with $(a_j,b_j) \in U$ for all $j \in \{1,\ldots,k\}$,
there is a homeomorphism $h : X \to X$ such that
\begin{itemize}
\item $h(a_j) = b_j$ for all $j \in \{1,\ldots,k\}$;
\item $(x,h(x)) \in V$ for all $x \in X$.
\end{itemize}
\end{remark}

Property (P) was called {\em property$^*$} in \cite{NKaw2019} in the setting of compact metric spaces.
It is well known that closed topological manifolds of dimension at least two have property (P)
(see \cite[Lemma~2.4.11]{NAokKHir94}, for instance)
and it was proved in \cite{NKaw2019} that any Cantor space $(X,d)$ has property (P).
In the sequel we shall prove that open convex sets in locally convex spaces of dimension at least two also have property (P).

\begin{theorem}\label{LCS-USLMH}
Every open convex set in a locally convex space of real dimension at least three is strongly locally multihomogeneous.
\end{theorem}

For the proof we will need the lemma below.

\begin{lemma}\label{LemmaHom}
Let $C$ be an open connected set in a locally convex space $X$. For any $a,b \in C$, there is a homeomorphism $h : X \to X$ such that
\[
h(a) = b \ \ \text{ and } \ \ h(x) = x \text{ for all } x \in X \backslash C.
\]
\end{lemma}

\begin{proof}
We shall divide the proof in five steps.

\smallskip\noindent
\textbf{ Step 1.} \textit{ Let $U\!:= \{x \in X : \|x - x_0\| < r\}$, where $\|\cdot\|$ is a continuous seminorm on $X$, $x_0 \in X$ and $r > 0$.
For any $a,b \in U$ with $\|a - x_0\| > 0$ and $b - x_0 = \lambda (a - x_0)$ for some $\lambda > 0$, there is a homeomorphism
$h : X \to X$ such that}
\[
h(a) = b \ \ \textit{ and } \ \ h(x) = x \textit{ for all } x \in (X \backslash U) \cup (x_0 + \Ker \|\cdot\|).
\]

\smallskip
It is enough to consider the case where $x_0 = 0$ and $r = 1$. Choose $k \in \N$ such that
\[
1 < \lambda (1 + k\|a\|) < 1 + k.
\]
Let $\phi : [1,\infty) \to [1,\infty)$ be the map that linearly interpolates the points
\[
(1,1), \ \ (1 + k\|a\|,\lambda (1 + k\|a\|)) \ \text{ and } \ (1 + k,1 + k),
\]
and equals the identity on $[1+k,\infty)$. We define $h : X \to X$ by
\[
h(x)\!:= \frac{\phi(1 + k\|x\|)}{1 + k\|x\|}\, x.
\]
Note that $h$ is continuous, $h(a) = b$ and $h(x) = x$ whenever $\|x\| = 0$ or $\|x\| \geq 1$.
In order to prove that $h$ is bijective, we consider the auxiliary function
\[
\psi: t \in [0,\infty) \mapsto \frac{\phi(1 + kt)\, t}{1 + kt} \in [0,\infty),
\]
which is a homeomorphism. We define $g : X \to X$ by
\[
g(x)\!:= x \ \text{ if } \|x\| = 0, \ \ \ g(x)\!:= \frac{\psi^{-1}(\|x\|)}{\|x\|}\, x \ \text{ if } \|x\| > 0.
\]
Simple computations show that $h(g(x)) = g(h(x)) = x$ for all $x \in X$, that is, $h$ is bijective and $g$ is its inverse.
It remains to prove that $g$ is continuous.
For this purpose, suppose that a net $(x_\alpha)_{\alpha \in I}$ in $X \backslash \Ker\|\cdot\|$
converges to a point $x \in \Ker\|\cdot\|$ in the topology of $X$.
We have to prove that $(g(x_\alpha))_{\alpha \in I}$ converges to $x$ in the topology of $X$. Since
\[
g(x_\alpha) = \frac{\psi^{-1}(\|x_\alpha\|)}{\|x_\alpha\|}\, x_\alpha \ \ \text{ for all } \alpha \in I,
\]
it is enough to show that
\[
\lim_{t \to 0^+} \frac{\psi^{-1}(t)}{t} = 1,
\]
which can be obtained from an application of L'H\^opital's rule.

\medskip\noindent
\textbf{ Step 2.} \textit{ Let $U\!:= \{x \in X : \|x - x_0\| < r\}$, where $\|\cdot\|$ is a continuous seminorm on~$X$, $x_0 \in X$ and $r > 0$.
For any $a \in U$ with $\|a - x_0\| > 0$, there is a homeomorphism $h : X \to X$ such that}
\[
h(a) = x_0 \ \ \textit{ and } \ \ h(x) = x \textit{ for all } x \in X \backslash U.
\]

\smallskip
Indeed, let
\[
t\!:= \frac{r - \|a - x_0\|}{3} > 0 \ \ \text{ and } \ \ r'\!:= r - t = \frac{2r + \|a - x_0\|}{3} > 0.
\]
Let
\[
x'_0\!:= x_0 - t \cdot \frac{a - x_0}{\|a - x_0\|} \ \ \text{ and } \ \ U'\!:= \{x \in X : \|x - x'_0\| < r'\}.
\]
Simple computations show that $x_0 \in U'$ and $U' \subset U$. Moreover,
\[
a - x'_0 = \frac{\|a - x_0\| + t}{\|a - x_0\|}\, (a - x_0) = \frac{\|a - x_0\| + t}{t}\, (x_0 - x'_0),
\]
which implies that $a \in U'$, $\|a - x'_0\| > 0$ and $x_0 - x'_0 = \lambda (a - x'_0)$, where
\[
\lambda\!:= \frac{t}{\|a - x_0\| + t} > 0.
\]
Hence, by Step~1, there is a homeomorphism $h : X \to X$ such that $h(a) = x_0$ and $h(x) = x$ for all $x \in X \backslash U'$.
Since $U' \subset U$, we are done.

\medskip\noindent
\textbf{ Step 3.} \textit{ Let $U\!:= \{x \in X : \|x - x_0\| < r\}$, where $\|\cdot\|$ is a continuous seminorm on~$X$, $x_0 \in X$ and $r > 0$.
For any $a,b \in U$, there is a homeomorphism $h : X \to X$ such that}
\[
h(a) = b \ \ \textit{ and } \ \ h(x) = x \textit{ for all } x \in X \backslash U.
\]

\smallskip
We may assume that $\|\cdot\|$ is not identically zero. Hence, for any $\eps > 0$, there exists $x'_0 \in X$ such that
$\|x'_0 - x_0\| < \eps$, $\|a - x'_0\| > 0$ and $\|b - x'_0\| > 0$. In particular,
\[
U'\!:= \{x \in X : \|x - x'_0\| < r - \eps\} \subset U.
\]
Moreover, by choosing $\eps > 0$ small enough, we have that $a,b \in U'$.
Thus, by Step 2, there are homeomorphisms $f : X \to X$ and $g : X \to X$ such that
\[
f(a) = x'_0, \ \ g(b) = x'_0 \ \text{ and } \ f(x) = g(x) = x \text{ for all } x \in X \backslash U'.
\]
Hence, $h\!:= g^{-1} \circ f$ does the job.

\medskip\noindent
\textbf{ Step 4.} \textit{ Proof of the case where $C$ is convex.}

\smallskip
Let $L$ be the line segment joining $a$ and $b$, that is,
\[
L\!:= \{(1-t)a+tb : t \in [0,1]\}.
\]
Since we are assuming that $C$ is convex, $L \subset C$.
Since $L$ is compact, there is a continuous seminorm $\|\cdot\|$ on $X$ such that the open convex neighborhood
\[
V\!:= \{x \in X : \|x\| < 1\}
\]
of $0$ in $X$ satisfies
\[
L + V \subset C.
\]
By compactness, there are finitely many points $c_1,\ldots,c_k \in L$ such that
\[
L \subset (c_1+V) \cup \ldots \cup (c_k+V).
\]
By reordering the $c_j$'s in a suitable way, if necessary, we may assume that $a \in c_1+V$, $b \in c_k+V$ and,
for each $1 \leq j \leq k-1$, we can choose a point
\[
d_j \in (c_j+V) \cap (c_{j+1}+V) \cap L.
\]
By Step~3, there are homeomorphisms $h_1,\ldots,h_k$ from $X$ onto itself such that
\[
h_1(a) = d_1, \ h_2(d_1) = d_2, \ \ldots, \ h_{k-1}(d_{k-2}) = d_{k-1}, \ h_k(d_{k-1}) = b, \text{ and}
\]
\[
h_j(x) = x \ \ \text{ for all } x \in X \backslash (c_j + V) \ \ \ (1 \leq j \leq k).
\]
Consequently, the homeomorphism $h\!:= h_k \circ \cdots \circ h_1$ has the desired properties.

\medskip\noindent
\textbf{ Step 5.} \textit{ Proof of the general case.}

\smallskip
Since $C$ is open and connected, there is a piecewise linear path $\psi : [0,1] \to X$ such that
\[
\psi(0) = a, \ \ \psi(1) = b \ \text{ and } \ P\!:= \psi([0,1]) \subset C.
\]
We can write $P = L_1 \cup \ldots \cup L_k$, where each $L_j$ is a line segment, $a$ is the initial point of $L_1$,
$b$ is the final point of $L_k$ and, for each $j \in \{1,\ldots,k-1\}$, the final point $d_j$ of $L_j$ coincides with the initial point of $L_{j+1}$.
Let $V$ be an open convex neighborhood of $0$ in $X$ such that $P + V \subset C$.
Since each set $L_j + V$ is open and convex, Step~4 gives us homeomorphisms $h_1,\ldots,h_k$ from $X$ onto itself such that
\[
h_1(a) = d_1, \ h_2(d_1) = d_2, \ \ldots, \ h_{k-1}(d_{k-2}) = d_{k-1}, \ h_k(d_{k-1}) = b, \text{ and}
\]
\[
h_j(x) = x \ \ \text{ for all } x \in X \backslash (L_j + V) \ \ \ (1 \leq j \leq k).
\]
Hence, $h\!:= h_k \circ \cdots \circ h_1$ is the homeomorphism we were looking for.
\end{proof}

We are now in position to prove Theorem~\ref{LCS-USLMH}.

\begin{proof}
Let $C$ be an open convex set in a locally convex space $X$ of real dimension at least three.
Recall that a basis for the uniformity of $C$ (induced by that of $X$) is given by the sets
\[
\wt{V}\!:= \{(x,y) \in C \times C : y - x \in V\},
\]
as $V$ runs through the set of all neighborhoods of $0$ in $X$. Note that
\[
\wt{V}(x) = (x + V) \cap C \ \ \text{ for all } x \in C.
\]

Given a neighborhood $V$ of $0$ in $X$, choose a neighborhood $W$ of $0$ in $X$ with $W + W \subset V$.
Let $U$ be an open absolutely convex neighborhood of $0$ in $X$ with $U + U + U + U \subset W$.
Take $k \geq 1$, $a_1,\ldots,a_k \in C$ pairwise distinct and $b_1,\ldots,b_k \in C$ pairwise distinct with $(a_j,b_j) \in \wt{U}$
for all $j \in \{1,\ldots,k\}$. We have to find a homeomorphism $h : C \to C$ with the following properties:
\begin{itemize}
\item [\rm (a)] $h(a_j) = b_j$ for all $j \in \{1,\ldots,k\}$;
\item [\rm (b)] $h(\wt{U}(a_j)) \subset \wt{V}(a_j)$ for all $j \in \{1,\ldots,k\}$;
\item [\rm (c)] $h(x) = x$ for all $x \in C \backslash \big(\wt{U}(a_1) \cup \ldots \cup \wt{U}(a_k)\big)$.
\end{itemize}
Let $U'$ be an open absolutely convex neighborhood of $0$ in $X$ such that $U' \subset U$,
\[
b_j + U' \subset \wt{U}(a_j) \ \ \text{ for all } j \in \{1,\ldots,k\},
\]
and the sets $b_1 + U',\ldots,b_k + U'$ are pairwise disjoint.
Choose points $b'_j \in b_j + U'$, $j \in \{1,\ldots,k\}$, so that the points $a_1,\ldots,a_k,b'_1,\ldots,b'_k$ are pairwise distinct.
Since $\dim_\R X \geq 3$, if $D$ is an open connected set in $X$ and $\{L_1,\ldots,L_m\}$ is a finite set
of closed line segments in $X$, then the set $D \backslash (L_1 \cup \ldots \cup L_m)$ is also open and connected
(this is not necessarily true if $\dim_\R X = 2$). Thus, there exist piecewise linear paths
\[
\psi_j : [0,1] \to X, \ \ \ j \in \{1,\ldots,k\},
\]
such that
\[
\psi_j(0) = a_j, \ \ \psi_j(1) = b'_j, \ \ P_j\!:= \psi_j([0,1]) \subset \wt{U}(a_j) \ \ \ (j \in \{1,\ldots,k\})
\]
and the sets $P_1,\ldots,P_k$ are pairwise disjoint. Let $Z$ be an open convex neighborhood of $0$ in $X$ such that
\[
C_j\!:= P_j + Z \subset \wt{U}(a_j) \ \ \text{ for all } j \in \{1,\ldots,k\},
\]
and the sets $C_1,\ldots,C_k$ are pairwise disjoint. Since $C_j$ is open and connected,
Lemma~\ref{LemmaHom} guarantees that there is a homeomorphism $g_j : X \to X$ such that
\[
g_j(a_j) = b'_j \ \ \text{ and } \ \ g_j(x) = x \text{ for all } x \in X \backslash C_j.
\]
Consider the homeomorphism $g\!:= g_k \circ \cdots \circ g_1 : X \to X$.
Also by Lemma~\ref{LemmaHom}, there is a homeomorphism $f_j : X \to X$ such that
\[
f_j(b'_j) = b_j \ \ \text{ and } \ \ f_j(x) = x \text{ for all } x \in X \backslash (b_j + U').
\]
Consider the homeomorphism $f\!:= f_k \circ \cdots \circ f_1 : X \to X$. Since
\[
C_1 \cup \ldots \cup C_k \subset \wt{U}(a_1) \cup \ldots \cup \wt{U}(a_k) \subset C,
\]
it follows from the construction of $g$ that
\[
g(x) = x \ \ \text{ for all } x \in X \backslash C.
\]
A similar reasoning shows that
\[
f(x) = x \ \ \text{ for all } x \in X \backslash C.
\]
Hence, the homeomorphism $f \circ g : X \to X$ maps $C$ onto $C$, and so it induces a homeomorphism $h : C \to C$.
It is clear that properties (a) and (c) hold.
We claim that
\[
g(x) \in x + U + U \ \text{ and } \ f(x) \in x + U' + U' \ \text{ for all } x \in C.
\]
Indeed, if $x \not\in C_1 \cup \ldots \cup C_k$, then $g(x) = x \in x + U + U$.
If $x \in C_j$ for some (necessarily unique) $j \in \{1,\ldots,k\}$, then
\[
g(x) = g_j(x) \in C_j \subset \wt{U}(a_j).
\]
Since $x \in \wt{U}(a_j)$, we conclude that $g(x) \in x + U + U$.
The proof in the case of the map $f$ is analogous. Therefore,
\[
h(x) = f(g(x)) \in (x + U + U + U + U) \cap C \subset \wt{W}(x).
\]
This implies property (b) and completes the proof.
\end{proof}

\begin{corollary}\label{LCS-P}
Every open convex set in a locally convex space of real dimension at least two has property (P).
\end{corollary}

\begin{proof}
Let $C$ be an open convex set in a locally convex space $X$ of real dimension at least two.
If $\dim_\R X \geq 3$, then the result follows from the previous theorem.
If $\dim_\R X = 2$, then we may assume that $X = \R^2$ endowed with the max norm.
Given $\delta > 0$, consider the following collections of pairwise disjoint open convex sets:
\begin{align*}
\cC_1\!&:= \big\{\big(\,]2n\delta,(2n+2)\delta[ \ \times \ ]2m\delta,(2m+2)\delta[\,\big) \cap C : n,m \in \Z\big\},\\
\cC_2\!&:= \big\{\big(\,](2n+1)\delta,(2n+3)\delta[ \ \times \ ]2m\delta,(2m+2)\delta[\,\big) \cap C : n,m \in \Z\big\},\\
\cC_3\!&:= \big\{\big(\,]2n\delta,(2n+2)\delta[ \ \times \ ](2m+1)\delta,(2m+3)\delta[\,\big) \cap C : n,m \in \Z\big\},\\
\cC_4\!&:= \big\{\big(\,](2n+1)\delta,(2n+3)\delta[ \ \times \ ](2m+1)\delta,(2m+3)\delta[\,\big) \cap C : n,m \in \Z\big\}.
\end{align*}
Take $k \geq 1$, $a_1,\ldots,a_k \in C$ pairwise distinct and $b_1,\ldots,b_k \in C$ pairwise distinct with $\|a_j - b_j\| < \delta$
for all $j \in \{1,\ldots,k\}$. An argument used in the previous proof shows that it is enough to consider the case in which
$a_1,\ldots,a_k,b_1,\ldots,b_k$ are pairwise distinct. Let
\begin{align*}
I_1\!&:= \{j \in \{1,\ldots,k\} : \text{some } A \in \cC_1 \text{ contains both } a_j \text{ and } b_j\},\\
I_2\!&:= \{j \in \{1,\ldots,k\} \backslash I_1 : \text{some } A \in \cC_2 \text{ contains both } a_j \text{ and } b_j\},\\
I_3\!&:= \{j \in \{1,\ldots,k\} \backslash (I_1 \cup I_2) : \text{some } A \in \cC_3 \text{ contains both } a_j \text{ and } b_j\},\\
I_4\!&:= \{j \in \{1,\ldots,k\} \backslash (I_1 \cup I_2 \cup I_3) : \text{some } A \in \cC_4 \text{ contains both } a_j \text{ and } b_j\}.
\end{align*}
Then $I_1 \cup \ldots \cup I_4 = \{1,\ldots,k\}$.
Now, for each $\ell \in \{1,\ldots,4\}$, we consider a homeomorphism $h_\ell : C \to C$ in the following way:
if $A \in \cC_\ell$ does not contain any pair $a_j,b_j$ with $j \in I_\ell$, then $h_\ell$ is the identity map on $\ov{A}$
(closure relative to $C$); if $A \in \cC_\ell$ contains such a pair $a_j,b_j$, then $h_\ell$ maps $\ov{A}$ onto itself,
maps $a_j$ to $b_j$ for each $a_j,b_j \in A$ with $j \in I_\ell$ and equals the identity on the boundary of $A$ (relative to $C$)
and at each of the other $a_i$'s and $b_i$'s.
The existence of $h_\ell$ in the second case follows from the fact that we can join $a_j$ and $b_j$ by a piecewise linear path $P_j$
contained in $A$, for each $j \in I_\ell$ with $a_j,b_j \in A$, in such a way that the paths $P_j$'s are pairwise disjoint and do not contain
any of the other $a_i$'s or $b_i$'s.
Then, we can choose $\eps > 0$ small enough so that the sets
\[
C_j\!:= P_j + \big(]-\eps,\eps[ \ \times \ ]-\eps,\eps[\big)
\]
are contained in $A$, are pairwise disjoint, and do not contain any of the other $a_i$'s or $b_i$'s.
Finally, we can argue as in the construction of the map $g$ in the previous proof to obtain $h_\ell : \ov{A} \to \ov{A}$
with the desired properties.
Now, the homeomorphism $h\!:= h_4 \circ \cdots \circ h_1 : C \to C$ maps $a_j$ to $b_j$ for all $j$ and satisfies
$\|h(x) - x\| < 8\delta$ for all $x \in C$. This proves that $C$ has property (P).
\end{proof}

\begin{remark}
In the case of complex scalars, Corollary~\ref{LCS-P} holds without any restriction on the dimension of the space.
But in the case of real scalars, the result is clearly false if the dimension of the space is one.
\end{remark}

Let us now define the concepts of finite shadowing and strict periodic shadowing in the setting of uniform spaces.
Let $X$ be a uniform space with uniformity $\cU$ and let $f : X \to X$ be a continuous map.
Given $U \in \cU$, a finite sequence $(x_j)_{j=0}^k$ is said to be a {\em $U$-chain} (resp.\ a {\em $U$-cycle}) for $f$ if
\[
(f(x_j),x_{j+1}) \in U \ \ \text{ for all } j \in \{0,\ldots,k-1\} \ \ \ (\text{resp.\ and } x_k = x_0).
\]
The map $f$ has the {\em finite shadowing property} (resp.\ the {\em strict periodic shadowing property})
if for every $V \in \cU$, there exists $U \in \cU$ such that for any $U$-chain (resp.\ $U$-cycle) $(x_j)_{j=0}^k$ for $f$,
there exists $a \in X$ with
\[
(x_j,f^j(a)) \in V \ \ \text{ for all } j \in \{0,\ldots,k\} \ \ \ (\text{resp.\ and } f^k(a) = a).
\]

The classical notion of {\em topological stability} for homeomorphisms on compact metric spaces can be extended
to the uniform space setting in the following natural way.

\begin{definition}\label{TS-US}
Let $X$ be a uniform space with uniformity $\cU$. A homeomorphism $h : X \to X$ is said to be {\em topologically stable} if
for every $V \in \cU$, there exists $U \in \cU$ such that for any homeomorphism $g : X \to X$ with
\[
(h(x),g(x)) \in U \ \ \text{ for all } x \in X,
\]
there is a continuous map $\phi : X \to X$ satisfying
\[
h \circ \phi = \phi \circ g \ \ \ \text{ and } \ \ \ (x,\phi(x)) \in V \text{ for all } x \in X.
\]
\end{definition}

The concept of topological stability was introduced by Walters~\cite{PWal1970} in 1970.
Later Walters~\cite{PWal1978} established the following result:
\textit{ If $X$ is a closed topological manifold of dimension at least two and $h : X \to X$ is a topologically stable homeomorphism,
then $h$ has the shadowing property.}

Recently, Kawaguchi~\cite{NKaw2019} observed that Walters' original arguments can be used to prove the following more general result:
\textit{ If $X$ is a perfect compact metric space with property (P) and $h : X \to X$ is a topologically stable homeomorphism,
then $h$ has the shadowing property and the strict periodic shadowing property.}

It is an amazing fact that Walters' original arguments can actually be adapted to establish the following very general result.

\begin{theorem}\label{P-FSP}
Let $X$ be a Hausdorff uniform space with uniformity $\cU$. If $X$ has property (P) and no isolated point,
then every topologically stable homeomorphism $h : X \to X$ has the finite shadowing property and the strict periodic shadowing property.
\end{theorem}

Recall that the notions of shadowing and finite shadowing coincide for homeomorphisms on compact metric spaces
\cite[Lemma~1.1.1]{SPil99}. With this basic result in hand, we see that Theorem~\ref{P-FSP} generalizes the above-mentioned results.

\begin{proof}
Fix $V \in \cU$ and choose $V' \in \cU$ with $V' \circ V' \subset V$.
Take a symmetric $W' \in \cU$ associated to $V'$ according to the topological stability of $h$.
Let $U' \in \cU$ be associated to $W'$ according to the fact that $X$ has property (P)
and take a symmetric $U \in \cU$ with $U \circ U \circ U \subset U'$.
Given a $U$-chain $(x_j)_{j=0}^k$ for $h$, we will find a point $a \in X$ such that
\[
(x_j,h^j(a)) \in V \ \ \text{ for all } j \in \{0,\ldots,k\}.
\]
By the continuity of $h$, there exists $Z \in \cU$, $Z \subset U \cap V'$, such that
\[
h(Z(x_j)) \subset U(h(x_j)) \ \ \text{ for all } j \in \{0,\ldots,k\}.
\]
Since $X$ is a Hausdorff space without isolated points, we can choose points $x'_j \in Z(x_j)$, $j \in \{0,\ldots,k\}$, so that
$x'_0,\ldots,x'_k$ are pairwise distinct. Since
\[
(h(x'_j),x'_{j+1}) \in U' \ \ \text{ for all } j \in \{0,\ldots,k-1\},
\]
property (P) gives us a homeomorphism $f : X \to X$ such that
\[
f(h(x'_j)) = x'_{j+1} \ \ \text{ for all } j \in \{0,\ldots,k-1\}
\]
and
\[
(x,f(x)) \in W' \ \ \text{ for all } x \in X.
\]
Consider the homeomorphism $g\!:= f \circ h$.
By topological stability, there is a continuous map $\phi : X \to X$ such that
\[
h \circ \phi = \phi \circ g \ \ \ \text{ and } \ \ \ (x,\phi(x)) \in V' \text{ for all } x \in X.
\]
Put $a\!:= \phi(x'_0) \in X$. Since $g^j(x'_0) = x'_j$, we get
\[
(x_j,h^j(a)) = (x_j,\phi(g^j(x'_0))) = (x_j,\phi(x'_j)) \in V \ \ \ (j \in \{0,\ldots,k\}).
\]
This proves that $h$ has the finite shadowing property.

In the case in which $(x_j)_{j=0}^k$ is a $U$-cycle for $h$, the points $x'_j \in Z(x_j)$, $j \in \{0,\ldots,k\}$,
can be chosen so that $x'_0,\ldots,x'_{k-1}$ are pairwise distinct and $x'_k = x'_0$.
With this choice, we have that the point $a$ has the following additional property:
\[
h^k(a) = h^k(\phi(x'_0)) = \phi(g^k(x'_0)) = \phi(x'_k) = a.
\]
This proves that $h$ has the strict periodic shadowing property.
\end{proof}

In view of Corollary~\ref{LCS-P} and Theorem~\ref{P-FSP}, we obtain the following result.

\begin{corollary}
Let $C$ be an open convex set in a locally convex space $X$ of real dimension at least two. If a homeomorphism
$h : C \to C$ is topologically stable, then $h$ has the finite shadowing property and the strict periodic shadowing property.
\end{corollary}

Theorem~\ref{TSO-FSP} is a special case of the above corollary in the case of spaces of real dimension greater than $1$.
But in the case of the real line $\R$, it is easy to see that if $T \in GL(\R)$ is topologically stable, then $T$ must be hyperbolic,
and so it has the shadowing property and the strict periodic shadowing property.
This completes the proof of Theorem~\ref{TSO-FSP}.

Our next goal is to investigate the topological stability of generalized hyperbolic operators.
For this purpose, recall that a subset $E$ of a locally convex space $X$ is said to be {\em bounded} if
for each neighborhood $V$ of $0$ in $X$, there exists $r > 0$ such that $E \subset \lambda V$ whenever $|\lambda| \geq r$.
If $(\|\cdot\|_\alpha)_{\alpha \in I}$ is a family of seminorms inducing the topology of $X$, then this is equivalent to say that
$E$ is bounded with respect to $\|\cdot\|_\alpha$ for every $\alpha \in I$.

Given locally convex spaces $X$ and $Y$, we denote by $F_b(X;Y)$ (resp.\ $C_b(X;Y)$, $U_b(X;Y)$) the vector space of
all bounded maps (resp.\ all bounded continuous maps, all bounded uniformly continuous maps) from $X$ into $Y$.
The next lemma generalizes Step~1 of the proof of \cite[Theorem~1]{NBerAMes20}.
The proof given below is based on the original proof from \cite{NBerAMes20}.

\begin{lemma}\label{Step1}
Let $X$ be a sequentially complete locally convex space. Let $T \in GL(X)$ be a generalized hyperbolic operator, let
\[
X = M \oplus N
\]
be the direct sum decomposition given by the definition of generalized hyperbolicity,
and let $P_M : X \to M$ and $P_N : X \to N$ be the canonical projections.
Consider the closed subspace $Y\!:= M + T^{-1}(N)$ of $X$.
For any bijective map $R : X \to X$, the linear map
\[
\Psi: \varphi \in F_b(X;Y) \mapsto \varphi \circ R - T \circ \varphi \in F_b(X;X)
\]
is bijective and its inverse is given by
\begin{equation}\label{Step1-1}
\Psi^{-1}(\phi)(x) = \sum_{k=0}^\infty T^k P_M(\phi(R^{-k-1}x)) - \sum_{k=1}^\infty T^{-k} P_N(\phi(R^{k-1}x)).
\end{equation}
Moreover, if $R$ is a homeomorphism (resp.\ a uniform homeomorphism), then
\[
\Psi(C_b(X;Y)) = C_b(X;X) \ \ \ (\text{resp.} \ \Psi(U_b(X;Y)) = U_b(X;X)).
\]
\end{lemma}

\begin{proof}
Let $(\|\cdot\|_\alpha)_{\alpha \in I}$ be a directed family of seminorms inducing the topology of $X$.
By the definition of generalized hyperbolicity, for each $\alpha \in I$,
there exist $\beta_\alpha \in I$, $c_\alpha > 0$ and $t_\alpha \in (0,1)$ such that
\[
\|T^n y\|_\alpha \leq c_\alpha\, t_\alpha^n\, \|y\|_{\beta_\alpha} \ \ \text{ and } \ \
\|T^{-n} z\|_\alpha \leq c_\alpha\, t_\alpha^n\, \|z\|_{\beta_\alpha},
\]
whenever $y \in M$, $z \in N$ and $n \in \N_0$. We shall now divide the proof in three steps.

\medskip\noindent
\textbf{ Step 1.} $\Psi$ is injective.

\medskip
Suppose that $\varphi \in F_b(X;Y)$ and $\Psi(\varphi) = 0$. Then $\varphi \circ R^n = T^n \circ \varphi$ for all $n \in \N$,
which gives the following formulas
\begin{equation}\label{Step1-3}
\varphi(x) = T^{-n}(\varphi(R^nx)) \ \ \text{ and } \ \ \varphi(x) = T^n(\varphi(R^{-n}x)) \ \ \ \ \ (x \in X, n \in \N).
\end{equation}
By the first of these equalities, $\varphi(x) = y_n(x) + z_n(x)$, where
\[
y_n(x) = T^{-n}P_M(\varphi(R^nx)) \ \ \text{ and } \ \ z_n(x) = T^{-n}P_N(\varphi(R^nx)).
\]
Clearly, $z_n(x) \in T^{-1}(N)$ for all $x \in X$ and $n \in \N$. Write
\[
y_n(x) = a_n(x) + b_n(x) \ \text{ with } a_n(x) \in M \text{ and } b_n(x) \in N.
\]
We claim that $y_n(x) \in M$ (i.e., $b_n(x) = 0$) for all $x$ and $n$. For $n = 1$,
\[
T(b_1(x)) = T(y_1(x)) - T(a_1(x)) = P_M(\varphi(Rx)) - T(a_1(x)) \in M.
\]
Since $\varphi(x) = a_1(x) + (b_1(x) + z_1(x))$, $a_1(x) \in M$ and $b_1(x) + z_1(x) \in N$,
the fact that $\varphi(x) \in Y$ implies that $b_1(x) \in T^{-1}(M) \cap T^{-1}(N) = \{0\}$.
Suppose that, for a certain $n \geq 1$, $b_n(x) = 0$ for all $x$. Then,
\[
T(b_{n+1}(x)) = T(y_{n+1}(x)) - T(a_{n+1}(x)) = y_n(Rx) - T(a_{n+1}(x)) \in M.
\]
By arguing as above, we conclude that $b_{n+1}(x) = 0$. By induction, our claim is proved.
This implies that $P_N(\varphi(x)) = z_n(x)$ for all $n \in \N$. Hence, for each $\alpha \in I$,
\[
\|P_N(\varphi(x))\|_\alpha \leq c_\alpha t_\alpha^n \|P_N(\varphi(R^nx))\|_{\beta_\alpha} \to 0 \ \text{ as } n \to \infty,
\]
since $P_N(\varphi(X))$ is a bounded subset of $X$. Thus,
\begin{equation}\label{Step1-4}
P_N(\varphi(x)) = 0 \ \ \text{ for all } x \in X.
\end{equation}
By the second equality in (\ref{Step1-3}), $\varphi(x) = y'_n(x) + z'_n(x)$, where
\[
y'_n(x) = T^nP_M(\varphi(R^{-n}x)) \ \ \text{ and } \ \ z'_n(x) = T^nP_N(\varphi(R^{-n}x)).
\]
Clearly, $y'_n(x) \in M$ for all $x$ and $n$.
Since $\varphi(R^{-1}x) \in Y$, we have that $P_N(\varphi(R^{-1}x)) \in T^{-1}(N)$, that is, $z'_1(x) = TP_N(\varphi(R^{-1}x)) \in N$.
Hence, $z'_1(x) \in T^{-1}(N)$. If $n \geq 1$ and $z'_n(x) \in T^{-1}(N)$ for all $x$, then $z'_{n+1}(x) = T(z'_n(R^{-1}x)) \in N$, and so
$z'_{n+1}(x) \in T^{-1}(N)$. By induction, $z'_n(x) \in T^{-1}(N)$ for all $x$ and $n$.
Thus, $P_M(\varphi(x)) = y'_n(x)$ for all $n \in \N$. As in the proof of (\ref{Step1-4}), we obtain
\begin{equation}\label{Step1-5}
P_M(\varphi(x)) = 0 \ \ \text{ for all } x \in X.
\end{equation}
By (\ref{Step1-4}) and (\ref{Step1-5}), $\varphi = 0$, proving the injectivity of $\Psi$.

\medskip\noindent
\textbf{ Step 2.} $\Psi$ is surjective (hence bijective) and its inverse is given by (\ref{Step1-1}).

\medskip
Take $\phi \in F_b(X;X)$. Since $P_M(\phi(X))$ and $P_N(\phi(X))$ are bounded subsets of $X$,
\[
\Big\|\sum_{k=i}^j T^k P_M(\phi(R^{-k-1}x))\Big\|_\alpha
  \leq c_\alpha \sum_{k=i}^j t_\alpha^k \|P_M(\phi(R^{-k-1}x))\|_{\beta_\alpha} \to 0 \ \text{ as } i,j \to \infty
\]
and
\[
\Big\|\sum_{k=i}^j T^{-k} P_N(\phi(R^{k-1}x))\Big\|_\alpha
  \leq c_\alpha \sum_{k=i}^j t_\alpha^k \|P_N(\phi(R^{k-1}x))\|_{\beta_\alpha} \to 0 \ \text{ as } i,j \to \infty.
\]
Hence, by the sequential completeness of $X$, we can define
\[
\varphi(x)\!:= \sum_{k=0}^\infty T^k P_M(\phi(R^{-k-1}x)) - \sum_{k=1}^\infty T^{-k} P_N(\phi(R^{k-1}x)) \ \ \ \ (x \in X).
\]
Note that the first series lies in $M$ and the second one lies in $T^{-1}(N)$.
Moreover, the above estimates imply that $\varphi(X)$ is a bounded subset of $X$.
So, $\varphi \in F_b(X;Y)$. Finally,
\begin{align*}
\Psi(\varphi)(x) &= \varphi(Rx) - T(\varphi(x))\\
  &= \big(P_M(\varphi(Rx)) - TP_M(\varphi(x))\big) + \big(P_N(\varphi(Rx)) - TP_N(\varphi(x))\big)\\
  &= \Big(\sum_{k=0}^\infty T^k P_M(\phi(R^{-k}x)) - \sum_{k=0}^\infty T^{k+1}P_M(\phi(R^{-k-1}x))\Big)\\
  &\ \ \ \ + \Big(-\sum_{k=1}^\infty T^{-k} P_N(\phi(R^kx)) + \sum_{k=1}^\infty T^{-k+1}P_N(\phi(R^{k-1}x))\Big)\\
  &= P_M(\phi(x)) + P_N(\phi(x))\\
  &= \phi(x).
\end{align*}
Thus, $\Psi$ is surjective (hence bijective) and (\ref{Step1-1}) holds.

\medskip\noindent
\textbf{ Step 3.} The last assertion in the theorem holds.

\medskip
This follows from the definition of $\Psi$ and the fact that both series in (\ref{Step1-1}) are uniformly convergent on $X$
(in view of the estimates obtained in Step~2).
\end{proof}

\begin{theorem}\label{GH-TS}
Let $T \in GL(X)$ be a generalized hyperbolic operator, where $X$ is a sequentially complete locally convex space.
For every neighborhood $V$ of $0$ in $X$, there exists a neighborhood $U$ of $0$ in $X$ such that
for any homeomorphism (resp.\ any uniform homeomorphism) $S : X \to X$ satisfying
\begin{equation}\label{GH-TS-1}
T - S \text{ is a bounded map} \ \ \ \text{ and } \ \ \ Tx - Sx \in U \text{ for all } x \in X,
\end{equation}
there is a continuous map (resp.\ a uniformly continuous map) $\phi : X \to X$ with
\[
T \circ \phi = \phi \circ S \ \ \ \text{ and } \ \ \ \phi(x) - x \in V \text{ for all } x \in X.
\]
\end{theorem}

\begin{proof}
Let $(\|\cdot\|_\alpha)_{\alpha \in I}$ be a directed family of seminorms inducing the topology of $X$.
Let $X = M \oplus N$ be the direct sum decomposition given by the generalized hyperbolicity of $T$
and let $P_M : X \to M$ and $P_N : X \to N$ be the canonical projections.
Given a neighborhood $V$ of $0$ in $X$, there exist $\alpha \in I$ and $\eps > 0$ such that
\[
\{x \in X : \|x\|_\alpha < \eps\} \subset V.
\]
There exist $\beta \in I$, $c > 0$ and $t \in (0,1)$ such that
\begin{equation}\label{GH-TS-2}
\|T^n y\|_\alpha \leq c\, t^n \|y\|_{\beta} \ \ \text{ and } \ \ \|T^{-n} z\|_\alpha \leq c\, t^n \|z\|_{\beta},
\end{equation}
whenever $y \in M$, $z \in N$ and $n \in \N_0$.
By the continuity of the projections $P_M$ and $P_N$, there exist $\gamma \in I$ and $d > 0$ such that
\begin{equation}\label{GH-TS-3}
\|P_M x\|_\beta \leq d\, \|x\|_\gamma \ \ \text{ and } \ \ \|P_N x\|_\beta \leq d\, \|x\|_\gamma \ \ \ \ \ (x \in X).
\end{equation}
Let $\delta\!:= \frac{(1-t)\eps}{2cd}$ and consider the neighborhood $U\!:= \{x \in X : \|x\|_\gamma < \delta\}$ of $0$ in $X$.
Let $S : X \to X$ be a homeomorphism such that $T - S$ is a bounded map and
\begin{equation}\label{GH-TS-4}
Tx - Sx \in U \ \ \text{ for all } x \in X.
\end{equation}
By Lemma~\ref{Step1}, the linear map
\[
\Psi: \varphi \in C_b(X;Y) \mapsto \varphi \circ S - T \circ \varphi \in C_b(X;X)
\]
is bijective, where $Y\!:= M + T^{-1}(N)$. Let
\[
\phi\!:= I + \Psi^{-1}(T - S),
\]
which is a continuous map from $X$ into $X$. Since
\[
T - S = \Psi(\phi - I) = \Psi(\phi) - \Psi(I) = \phi \circ S - T \circ \phi + T - S,
\]
we get
\[
T \circ \phi = \phi \circ S.
\]
Now, by (\ref{Step1-1}),
\[
\phi(x) - x  = \sum_{k=0}^\infty T^k P_M((T-S)(S^{-k-1}x)) - \sum_{k=1}^\infty T^{-k} P_N((T-S)(S^{k-1}x)),
\]
for every $x \in X$. Hence, simple computations using (\ref{GH-TS-2}), (\ref{GH-TS-3}) and (\ref{GH-TS-4}) give
\[
\|\phi(x) - x\|_\alpha < \eps \ \ \text{ for all } x \in X,
\]
which implies that $\phi(x) - x \in V$ for all $x \in X$.

If $S$ is a uniform homeomorphism, then we can consider $\Psi$ as a bijection from $U_b(X;Y)$ onto $U_b(X;X)$;
hence, the map $\phi$ obtained above is uniformly continuous.
\end{proof}

If $X$ is a Banach space, then the first condition in (\ref{GH-TS-1}) is superfluous, and so we obtain the following result.

\begin{corollary}\label{GH-TS-Cor}
Every invertible generalized hyperbolic operator on a Banach space is topologically stable.
\end{corollary}

Given a metric space $M$ with metric $d$, recall that a homeomorphism $h : M \to M$ is said to be {\em expansive} if there is
a constant $c > 0$ (an {\em expansivity constant} for $h$) such that, for every $x,y \in M$ with $x \neq y$, there exists $n \in \Z$
with $d(h^n(x),h^n(y)) \geq c$.

A famous result known as {\em Walters' stability theorem} asserts that:
\begin{quote}
\textit{ If a homeomorphism of a compact metric space is expansive and has the shadowing property,
then it is topologically stable} \cite[Theorem~4]{PWal1978}.
\end{quote}
In linear dynamics, the following more precise result is known:
\begin{quote}
\textit{ For invertible operators on Banach spaces, expansivity plus the shadowing property is equivalent to hyperbolicity}
\cite[Theorem~1]{NBerAMes21}.
\end{quote}
In particular, expansive operators with the shadowing property on Banach spaces are topologically stable.
Moreover, by Corollaries~\ref{TSO-FSP-Cor} and~\ref{GH-TS-Cor}, we also obtain the following characterization of hyperbolicity.

\begin{corollary}
An invertible operator on a Banach space is hyperbolic if and only if it is expansive and topologically stable.
\end{corollary}

Since hyperbolicity and the shadowing property coincide for invertible operators in the finite-dimensional setting,
Corollaries~\ref{TSO-FSP-Cor} and~\ref{GH-TS-Cor} also imply the following fact.

\begin{corollary}
For invertible operators on finite-dimensional normed spaces, the concepts of hyperbolicity and topological stability coincide.
\end{corollary}

\begin{remark}
Let $X$, $T$, $M$, $N$ and $Y$ be as in the previous theorem and its proof.
The proof of the theorem shows that: For any homeomorphism $S : X \to X$ such that
\[
T - S \ \text{ is a bounded map},
\]
there exists a continuous map $\phi : X \to X$ such that
\begin{equation}\label{Unique}
T \circ \phi = \phi \circ S \ \ \ \text{ and } \ \ \ \phi - I \text{ is a bounded map}.
\end{equation}
Moreover, $\phi$ is unique as long as we require that $(\phi - I)(X) \subset Y$.
If $T$ is hyperbolic, then $Y = X$ and we have the uniqueness of the continuous map $\phi : X \to X$ satisfying (\ref{Unique}).
However, in the nonhyperbolic case, such a map $\phi : X \to X$ is never unique.
Indeed, if $T$ is not hyperbolic, then we can choose a nonzero vector $y \in T^{-1}(M) \cap N$. Hence,
\[
z\!:= \sum_{n=-\infty}^\infty T^ny
\]
defines a nontrivial fixed point of $T$ and the continuous map $\psi : X \to X$ given by $\psi(x)\!:= \phi(x) + z$ ($x \in X$)
also satisfies (\ref{Unique}).
\end{remark}

\begin{remark}
Let $T$ be the invertible generalized hyperbolic operator defined in Remark~\ref{CounterexCompleteness}.
Since $T$ does not have the periodic shadowing property, it follows from Theorem~\ref{TSO-FSP} that $T$ is not topologically stable.
This shows that the hypothesis of sequential completeness (resp.\ of completeness) is essential for the validity of Theorem~\ref{GH-TS}
(resp.\ of Corollary~\ref{GH-TS-Cor}).
\end{remark}


\section{Expansivity and uniform expansivity for operators}\label{ExpSection}

Given a seminorm $\|\cdot\|$ on a vector space $X$, we define the {\em unit sphere} of $\|\cdot\|$ by
\[
S_{\|\cdot\|}\!:= \{x \in X : \|x\| = 1\}.
\]
If $X$ is a normed space with norm $\|\cdot\|$, we also write $S_X$ instead of $S_{\|\cdot\|}$.

If $X$ is a normed space and $T \in GL(X)$, then it is well known that $T$ is expansive if and only if the following property holds:
\begin{itemize}
\item \textit{ For every $x \in S_X$, there exists $n \in \Z$ such that $\|T^nx\| \geq 2$.}
\end{itemize}
Actually, this property is often used as the definition of expansive operators on normed spaces and motivated the concept of
uniform expansivity. Recall that the operator $T$ is said to be {\em uniformly expansive} if the following property holds:
\begin{itemize}
\item \textit{ There exists $n \in \N$ such that, for every $x \in S_X$, $\|T^nx\| \geq 2$ or $\|T^{-n}x\| \geq 2$.}
\end{itemize}

We refer the reader to \cite{AlvBerMes21,BerCirDarMesPuj18,NBerAMes21,MEisJHed70,JHed71,MMaz00}
for the study of expansive and uniformly expansive operators on Banach spaces.

Let us now consider the following simple example in the Fr\'echet space setting.

\begin{example}\label{CE}
Let $C(\C)$ be the Fr\'echet space of all continuous maps $f : \C \to \C$ endowed with the compact-open topology,
which is induced by the seminorms
\[
\|f\|_k\!:= \max_{|z| \leq k} |f(z)| \ \ \ \ (k \in \N).
\]
A canonical compatible invariant metric on $C(\C)$ is given by
\[
d(f,g)\!:= \sum_{k=1}^\infty \frac{1}{2^k} \min\{1,\|f - g\|_k\}.
\]
Let $T\!:= 2I$ be twice the identity operator on $C(\C)$.
Since $d(T^nf,T^ng) = d(2^nf,2^ng) \leq \frac{1}{2^k}$ for all $n \in \Z$,
whenever $f,g \in C(\C)$ have supports disjoint from the disk $\{z \in \C : |z| \leq k\}$,
it follows that $T$ is not expansive.
\end{example}

Intuitively speaking, an operator of the form $2I$ should be expansive. Hence, the above example shows that the metric notion
of expansivity does not behave so well in the Fr\'echet space setting as it does in the Banach space setting.

Below we will propose an alternative notion of (uniform) expansivity for operators on locally convex spaces,
which is motivated by the following simple characterizations in the case of invertible operators on normed spaces,
as observed in \cite[Proposition~19]{BerCirDarMesPuj18}:
\begin{itemize}
\item $T$ is expansive $\Leftrightarrow$ $\sup_{n \in \Z} \|T^nx\| = \infty$ for every nonzero $x \in X$.
\item $T$ is uniformly expansive $\Leftrightarrow$ $S_X = A \cup B$ where
  $\lim_{n \to \infty} \|T^nx\| = \infty$ uniformly on $A$ and $\lim_{n \to \infty} \|T^{-n}x\| = \infty$ uniformly on $B$.
\end{itemize}

Our definition is also motivated by the fact that the important notion of {\em bounded set} in functional analysis is the topological one:
a subset $E$ of a topological vector space $X$ is said to be {\em bounded} if for each neighborhood $V$ of $0$ in $X$, there exists
$r > 0$ such that $E \subset \lambda V$ whenever $|\lambda| \geq r$.
If $X$ is a normed space, then the set $E$ is bounded if and only if it is bounded in the metric sense, i.e., it has finite diameter.
Nevertheless, the metric notion of bounded set becomes essentially useless when we go beyond the normed space setting,
while the topological notion plays a central role in all of functional analysis.
As we have seen above, in normed spaces, expansivity means that the orbit of each nonzero vector is an unbounded set.
These comments together with Example~\ref{CE} suggest that the metric notion of expansivity may not be the most appropriate one
in more general settings.

\begin{definition}\label{exp-uexp}
Let $X$ be a locally convex space over $\K$ whose topology is induced by a directed family $(\|\cdot\|_\alpha)_{\alpha \in I}$ of seminorms.
We say that an operator $T \in GL(X)$ is {\em topologically expansive} if the following condition holds:
\begin{itemize}
\item [(E)] For each nonzero $x \in X$, there exists $\alpha \in I$ such that $\sup_{n \in \Z} \|T^n x\|_\alpha = \infty$.
\end{itemize}
We say that the operator $T$ is {\em uniformly topologically expansive} if:
\begin{itemize}
\item [(UE)] For every $\alpha \in I$, there exists $\beta \in I$ such that we can write $S_{\|\cdot\|_\alpha} = A_\alpha \cup B_\alpha$, where
\[
\|T^n x\|_\beta \to \infty \text{ uniformly on } A_\alpha \text{ as } n \to \infty
\]
and
\[
\|T^{-n} x\|_\beta \to \infty \text{ uniformly on } B_\alpha \text{ as } n \to \infty.
\]
\end{itemize}
\end{definition}

Note that the above notions are independent of the choice of the directed family $(\|\cdot\|_\alpha)_{\alpha \in I}$
of seminorms inducing the topology of $X$.
Moreover, condition (E) means that the orbit $\Orb(x,T)\!:= \{T^n x : n \in \Z\}$ is an unbounded set in $X$ for every nonzero $x \in X$.

\begin{remark}
If $X$ is a {\em metrizable} locally convex space, then its topology is induced by an increasing sequence $(\|\cdot\|_k)_{k \in \N}$
of seminorms and a canonical compatible invariant metric on $X$ is given by
\[
d(x,y)\!:= \sum_{k=1}^\infty \frac{1}{2^k} \min\{1,\|x - y\|_k\}.
\]
In this case, it is easy to show that every expansive operator on $X$ is topologically expansive.
However, the converse is not true in general, since operators of the form $\lambda I$, with $|\lambda| \not\in \{0,1\}$,
are always uniformly topologically expansive but may fail to be expansive (Example~\ref{CE}).
\end{remark}

\begin{proposition}
Let $X$ be a locally convex space whose topology is induced by a directed family $(\|\cdot\|_\alpha)_{\alpha \in I}$ of seminorms.
An operator $T \in GL(X)$ is topologically expansive if and only if for every $x \in X \backslash \{0\}$, there exists $\alpha \in I$ such that
$\|x\|_\alpha \neq 0$ and, for each $y \in \Orb(x,T)$, there exists $n \in \Z$ with $\|T^ny\|_\alpha \geq 2\|y\|_\alpha$.
\end{proposition}

\begin{proof}
($\Rightarrow$): Let $x \in X \backslash \{0\}$ and let $\alpha \in I$ be as in (E).
Without loss of generality, we may assume that $\|x\|_\alpha \neq 0$.
Given $y \in \Orb(x,T)$, since the set $\{\|T^n y\|_\alpha : n \in \Z\}$ is unbounded in~$\R$,
there exists $n \in \Z$ with $\|T^ny\|_\alpha \geq 2 \|y\|_\alpha$.

\smallskip\noindent
($\Leftarrow$): Given $x \in X \backslash \{0\}$, let $\alpha \in I$ be as in the hypothesis.
By induction, we obtain a sequence $(n_j)_{j \in \N}$ in $\Z$ such that
\[
\|T^{n_1 + \cdots + n_j}x\|_\alpha \geq 2^j \|x\|_\alpha \ \ \text{ for all } j \in \N.
\]
Thus, $\{\|T^nx\|_\alpha : n \in \Z\}$ is unbounded in $\R$, giving (E).
\end{proof}

Note that the constant $2$ can be replaced by any constant $c > 1$ in the above proposition.

\begin{theorem}\label{GHEquiv}
For any invertible generalized hyperbolic operator $T$ on a locally convex space~$X$, the following assertions are equivalent:
\begin{itemize}
\item [\rm (i)] $T$ is topologically expansive;
\item [\rm (ii)] $T$ is uniformly topologically expansive;
\item [\rm (iii)] $T$ is hyperbolic.
\end{itemize}
\end{theorem}

\begin{proof}
Let $(\|\cdot\|_\alpha)_{\alpha \in I}$ be a directed family of seminorms that induces the topology of $X$.
Let $X = M \oplus N$ be the topological direct sum decomposition given by the generalized hyperbolicity of $T$
and let $P_M : X \to M$ and $P_N : X \to N$ be the canonical projections.

\smallskip\noindent
(iii) $\Rightarrow$ (ii): Given $\alpha \in I$, there exist $\beta \in I$, $c > 0$ and $t \in (0,1)$ such that
\[
\|T^{-n} y\|_\beta \geq c^{-1} t^{-n} \|y\|_\alpha \ \text{ and } \ \|T^n z\|_\beta \geq c^{-1} t^{-n} \|z\|_\alpha \
  \text{ for all } y \in M, z \in N, n \in \N_0.
\]
Choose $\gamma \in I$ and $d > 0$ such that
\[
\|P_M x\|_\beta \leq d\, \|x\|_\gamma \ \text{ and } \ \|P_N x\|_\beta \leq d\, \|x\|_\gamma \ \text{ for all } x \in X.
\]
Let
\[
A_\alpha\!:= \{x \in S_{\|\cdot\|_\alpha} : \|P_N x\|_\alpha \geq 1/2\} \ \text{ and } \
B_\alpha\!:= \{x \in S_{\|\cdot\|_\alpha} : \|P_M x\|_\alpha \geq 1/2\}.
\]
Then $S_{\|\cdot\|_\alpha} = A_\alpha \cup B_\alpha$ and
\[
\|T^n x_1\|_\gamma \geq \frac{1}{2\,c\,d\,t^n} \ \text{ and } \ \|T^{-n} x_2\|_\gamma \geq \frac{1}{2\,c\,d\,t^n} \
  \text{ for all } x_1 \in A_\alpha, x_2 \in B_\alpha, n \in \N_0,
\]
which proves that $T$ is uniformly topologically expansive.

\smallskip\noindent
(ii) $\Rightarrow$ (i): Clear.

\smallskip\noindent
(i) $\Rightarrow$ (iii): Take $x \in M \cap T(N)$ and $\alpha \in I$. Let $\beta \in I$, $c > 0$ and $t \in (0,1)$ be as in (GH3).
Since $x \in M$ and $T^{-1}x \in N$, we have that
\[
\|T^n x\|_\alpha \leq c\, t^n \|x\|_\beta \ \text{ and } \ \|T^{-n}T^{-1} x\|_\alpha \leq c\, t^n \|T^{-1}x\|_\beta \ \text{ for all } n \in \N_0.
\]
Hence,
\[
\sup_{n \in \Z} \|T^n x\|_\alpha \leq c \, \max\{\|x\|_\beta,\|T^{-1}x\|_\beta\} < \infty.
\]
Since $\alpha \in I$ is arbitrary and $T$ is topologically expansive, we conclude that $x = 0$.
This proves that $M \cap T(N) = \{0\}$, which implies that $T$ is hyperbolic.
\end{proof}

\begin{corollary}\label{h-ute}
Every invertible hyperbolic operator on a locally convex space is uniformly topologically expansive.
\end{corollary}

In the next section we will exhibit some examples of uniformly topologically expansive operators on the Fr\'echet space
$s(\Z)$ of rapidly decreasing sequences on $\Z$ that are not hyperbolic.
Examples of this type on the infinite-dimensional separable Hilbert space were given in \cite{MEisJHed70}.

The study of chaotic properties for linear operators, since the work of Birkhoff almost 100 years ago, 
derived in the well established theory of Linear Dynamics, which nowadays is very active. 
Once a dynamical notion is analyzed in the context of linear operators, 
it is natural to ask about its compatibility with classical chaotic properties. 
We want to finish this section by showing that uniform topological expansivity avoids the possibility of even the weakest properties
related to chaotic behavior, namely, topological transitivity and Li-Yorke chaos.

Given a metric space $M$ with metric $d$, recall that a continuous map $f : M \to M$ is said to be {\em Li-Yorke chaotic} if
there is an uncountable set $S \subset M$ such that each pair $(x,y)$ of distinct points in $S$ is a {\em Li-Yorke pair} for $f$,
in the sense that
\[
\liminf_{n \to \infty} d(f^n(x),f^n(y)) = 0 \ \ \text{ and } \ \ \limsup_{n \to \infty} d(f^n(x),f^n(y)) > 0.
\]

An extensive study of Li-Yorke chaos in the setting of linear dynamics was developed in \cite{BerBonMarPer11,BerBonMulPer15}.
In particular, the following useful characterizations were obtained:
\textit{ For any continuous linear operator $T$ on a Fr\'echet space $X$, the following assertions are equivalent:
\begin{itemize}
\item [\rm (i)] $T$ is Li-Yorke chaotic;
\item [\rm (ii)] $T$ admits a {\em semi-irregular vector}, that is, a vector $x \in X$ such that the sequence $(T^nx)_{n \in \N}$
  does not converge to zero but has a subsequence converging to zero;
\item [\rm (iii)] $T$ admits an {\em irregular vector}, that is, a vector $x \in X$ such that the sequence $(T^nx)_{n \in \N}$
  is unbounded but has a subsequence converging to zero.
\end{itemize}}

The notion of Li-Yorke chaos was extended to group actions on Hausdorff uniform spaces in \cite{TAra18}.
In the case of a continuous linear operator $T$ on a topological vector space~$X$, the definition reads as follows:
the operator $T$ is said to be {\em Li-Yorke chaotic} if there is an uncountable set $S \subset X$ such that each pair $(x,y)$
of distinct points in $S$ is a {\em Li-Yorke pair} for~$T$, in the sense that the following conditions hold:
\begin{itemize}
\item [(LY1)] For every neighborhood $V$ of $0$ in $X$, there exists $n \in \N$ such that $T^n x - T^n y \in V$.
\item [(LY2)] There exists a neighborhood $U$ of $0$ in $X$ such that $T^n x - T^n y \not\in U$ for infinitely many values of $n$.
\end{itemize}
It was observed in \cite{BCarVFav20} that the equivalence (i) $\Leftrightarrow$ (ii) remains true in this more general setting,
where $x$ {\em semi-irregular} for $T$ means that the sequence $(T^nx)_{n \in \N}$ does not converge to zero but has a
subnet converging to zero.

We are now in a position to show that uniform topological expansivity avoids the possibility of Li-Yorke chaos or topological transitivity.
In particular, the following result generalizes \cite[Theorem~C]{BerCirDarMesPuj18} from Banach spaces to arbitrary
locally convex spaces.

\begin{theorem}\label{ute-notLY}
A uniformly topologically expansive operator on a locally convex space is neither Li-Yorke chaotic nor topologically transitive.
\end{theorem}

\begin{proof}
We will divide the proof into two steps.

\smallskip\noindent
\textbf{ Step 1.} \textit{
If a continuous linear operator $T$ on a locally convex space $X$, whose topology is induced by a directed family
$(\|\cdot\|_\alpha)_{\alpha \in I}$ of seminorms, is topologically transitive or Li-Yorke chaotic, then it satisfies the following property:}
\begin{eqnarray}\nonumber\label{weak-trans-ly}
\exists \alpha_0 \in I, \  \forall \beta \in I, \ \exists \eps >0, \ \forall k \in \N, \ \exists m,n \in \N, \ \exists x \in X \mbox{ such that} \\
k < m < n-k, \ \| x\|_\beta < 1, \ \| T^nx\|_\beta < 1, \ \mbox{ and } \ \| T^mx\|_{\alpha_0} > \eps.
\end{eqnarray}

\smallskip
Suppose that $T$ is topologically transitive. Select any $\alpha_0 \in I$ such that $\| \cdot \|_{\alpha_0}$ is non-zero.
Given $\beta \in I$ arbitrary, we fix $\eps\!:= 1$, and once $k \in \N$ is fixed, by transitivity of $T$ we select
$m \in \N$, $m > k$, with $T^m(A) \cap B \neq \emptyset$, where
\[
A\!:=\{u \in X : \|u\|_\beta < 1\} \ \mbox{ and } \ B\!:= \{v \in X : \|v\|_{\alpha_0} > 1\}.
\]
Thus, $C\!:= A \cap T^{-m}(B)$ is a non-empty open set, and applying the transitivity property once again we find $n \in \N$,
$n > m+k$, with $T^n(C) \cap A \neq \emptyset$. Selecting any $x \in C \cap T^{-n}(A)$ yields property~\eqref{weak-trans-ly}.

In case that $T$ is Li-Yorke chaotic, fix any semi-irregular vector $u \in X$ for $T$.
In particular, there exist $\alpha_0 \in I$, $\delta \in\; ]0,1[$ and an increasing sequence $(m_j)_{j \in \N}$ of positive integers
such that $\| T^{m_j}u\|_{\alpha_0} > \delta$ for all $j \in \N$.
Given any $\beta \in I$, by semi-irregularity of $u$ we also find another increasing sequence $(n_j)_{j \in \N}$ of positive integers
such that $\| T^{n_j}u\|_{\beta} < 1$ for all $j \in \N$.
We pick $\lambda \in\; ]0,1[$ such that $\|\lambda u\|_{\beta} < 1$.
If we fix $\eps\!:= \delta \lambda$ and $x\!:= \lambda u$, then $\| x\|_\beta < 1$ and, for an arbitrary $k \in \N$,
we have the existence of $m > k$ with $\|T^{m}x\|_{\alpha_0} > \eps$ and also the existence of $n > m+k$ with
$\|T^nx\|_\beta < 1$, as desired.

\smallskip\noindent
\textbf{ Step 2.} \textit{No uniformly topologically expansive operator is Li-Yorke chaotic or topologically transitive.}
\smallskip

Let $T$ be a uniformly topologically expansive operator on a locally convex space $X$.
Let $(\|\cdot\|_\alpha)_{\alpha \in I}$ be a directed family of seminorms that induces the topology of~$X$.
Suppose that $T$ is Li-Yorke chaotic or topologically transitive.
Then it satisfies property~\eqref{weak-trans-ly} of Step 1.
Let $\alpha\!:=\alpha_0 \in I$ be given by this property.
By Definition~\ref{exp-uexp}, we can decompose $S_{\|\cdot\|_\alpha} = A_\alpha \cup B_\alpha$ and find $\beta \in I$
(w.l.o.g., $\beta \geq \alpha$) satisfying the conditions of uniform topological expansivity.
Let $\eps > 0$ be associated with $\alpha_0$ and $\beta$ in property~\eqref{weak-trans-ly}.
Since $\|T^ju\|_\beta \to \infty$ and $\|T^{-j}v\|_\beta \to \infty$, uniformly on $u \in A_\alpha$ and $v \in B_\alpha$, as $j \to \infty$,
we find $k \in \N$ with
\[
\min\{\| T^ju\|_\beta, \|T^{-j}v\|_\beta\} > 1/\eps
\]
for each pair $(u,v) \in A_\alpha \times B_\alpha$ and for any $j \geq k$.
By property~\eqref{weak-trans-ly}, there are $m,n \in \N$, $k < m < n-k$, and $x \in X$ so that
\[
\|x\|_\beta < 1, \ \|T^nx\|_\beta < 1, \ \mbox{ and } \ \|T^mx\|_{\alpha_0} > \eps.
\]
We set $y\!:= \left(\|T^mx\|_\alpha^{-1}\right) T^mx \in S_{\|\cdot\|_\alpha}$. If $y \in A_\alpha$, then
\[
\frac{1}{\eps} < \|T^{n-m}y\|_\beta = \frac{1}{\|T^mx\|_\alpha} \|T^nx\|_\beta < \frac{1}{\eps}\,,
\]
a contradiction. If $y \in B_\alpha$, then
\[
\frac{1}{\eps} < \|T^{-m}y\|_\beta = \frac{1}{\|T^mx\|_\alpha}\| x\|_\beta < \frac{1}{\eps}\,,
\]
which yields a contradiction too, and we conclude that $T$ cannot satisfy property~\eqref{weak-trans-ly},
so $T$ is neither Li-Yorke chaotic nor topologically transitive.
\end{proof}

\begin{remark}
Observe that property~\eqref{weak-trans-ly} is trivially satisfied for a unimodular multiple of the identity operator,
so it cannot be identified as a ``chaotic'' property, per se.
\end{remark}


\section{Weighted shifts on Fréchet sequence spaces}\label{TEWS}

Our main goal in this section is to explore the notion of topological expansivity for weighted shifts on Fr\'echet sequence spaces,
but some weighted shifts with other dynamical properties will be presented at the end of the section.

Let $\omega(\Z)\!:= \K^\Z$ be the Fréchet space of all scalar sequences equipped with the product topology.
Recall that the {\em bilateral forward} (resp.\ {\em backward}) {\em shift} on $\omega(\Z)$ is the continuous linear operator
$F : \omega(\Z) \to \omega(\Z)$ (resp.\ $B : \omega(\Z) \to \omega(\Z)$) given by
\[
F((x_n)_{n \in \Z})\!:= (x_{n-1})_{n \in \Z} \ \ \ \ \ (\text{resp. } B((x_n)_{n \in \Z})\!:= (x_{n+1})_{n \in \Z}).
\]
More generally, given a sequence $w\!:= (w_n)_{n \in \Z}$ of nonzero scalars (called a {\em weight sequence}), recall that
the {\em bilateral weighted forward} (resp.\ {\em backward}) {\em shift} on $\omega(\Z)$ is the continuous linear operator
$F_w : \omega(\Z) \to \omega(\Z)$ (resp.\ $B_w : \omega(\Z) \to \omega(\Z)$) given by
\[
F_w((x_n)_{n \in \Z})\!:= (w_{n-1}x_{n-1})_{n \in \Z} \ \ \ \ \ (\text{resp. } B_w((x_n)_{n \in \Z})\!:= (w_{n+1}x_{n+1})_{n \in \Z}).
\]
Note that the operators $F_w$ and $B_w$ are invertible on $\omega(\Z)$, and their inverses are given by
\[
F_w^{-1} = B_v \ \ \text{ and } \ \ B_w^{-1} = F_u,
\text{ where } v\!:= \Big(\frac{1}{w_{n-1}}\Big)_{n \in \Z} \text{ and } u\!:= \Big(\frac{1}{w_{n+1}}\Big)_{n \in \Z}.
\]

We will also consider these operators on a {\em Fréchet sequence space over $\Z$}, that is, a subspace $X$ of $\omega(\Z)$
equipped with a  topology under which $X$ is a Fréchet space and the embedding $X \hookrightarrow \omega(\Z)$ is continuous.
This continuity is equivalent to the continuity of each coordinate functional $(x_n)_{n \in \Z} \in X \mapsto x_j \in \K$, $j \in \Z$.
It follows from the closed graph theorem that $F_w$ (resp.\ $B_w$) induces a continuous linear operator on $X$ as soon as
it maps $X$ into itself. In this case, by abuse of notation, we also write $F_w : X \to X$ (resp.\ $B_w : X \to X$).
The canonical vectors $e_n\!:= (\delta_{j,n})_{j \in \Z}$, $n \in \Z$, of $\omega(\Z)$ form a {\em basis} in $X$ if they belong to $X$
and every sequence $x\!:= (x_n)_{n \in \Z} \in X$ satisfies
\[
x = \lim_{m,n \to \infty} (\ldots,0,0,x_{-m},x_{-m+1},\ldots,x_{n-1},x_{n},0,0,\ldots).
\]
This means that each $x \in X$ has a unique representation $x = \sum_{n \in \Z} x_n e_n$, with scalars $x_n \in \K$, $n \in \Z$.
Clearly $(e_n)_{n \in \Z}$ is a basis of the following Fréchet sequence spaces:
$\ell^p(\Z)$, $1 \leq p < \infty$, $c_0(\Z)$ and $\omega(\Z)$.

Our main goal in this section is to characterize topological expansivity for weighted shifts on Fr\'echet sequence spaces.
In order to state our results, we will need the following definition.

\begin{definition}\label{topologicallyEXPANsive}
Let $X$ be a locally convex space over $\K$ whose topology is induced by a directed family $(\|\cdot\|_\alpha)_{\alpha \in I}$
of seminorms. We say that an operator $T \in L(X)$ is {\em topologically positively expansive} if the following condition holds:
\begin{itemize}
\item [(PE)] For each nonzero $x \in X$, there exists $\alpha \in I$ such that $\sup_{n \in \N} \|T^nx\|_\alpha = \infty$.
\end{itemize}
\end{definition}

\begin{theorem}\label{shift_space}
Suppose that $X$ is a Fréchet sequence space over $\Z$ in which the sequence $(e_n)_{n \in \Z}$ of canonical vectors is a basis,
$(\|\cdot\|_k)_{k \in \N}$ is an increasing sequence of seminorms that induces the topology of $X$,
and the bilateral forward shift $F$ is an invertible operator on $X$. Then the following assertions are equivalent:
\begin{itemize}
\item [\rm (i)] $F : X \to X$ is topologically expansive;
\item [\rm (ii)] there exists $k \in \N$ such that
    \begin{itemize}
    \item [\rm (a)] $\sup\limits_{n \in \N} \|e_{n}\|_{k} = \infty$ or
    \item [\rm (b)] $\sup\limits_{n \in \N}\| e_{-n}\|_{k} = \infty$;
    \end{itemize}
\item [\rm (iii)]
    \begin{itemize}
    \item [\rm (a)] $F : X \to X$ is topologically positively expansive or
    \item [\rm (b)] $F^{-1} : X \to X$ is topologically positively expansive.
    \end{itemize}
\end{itemize}
\end{theorem}

\begin{proof}
(i) $\Rightarrow$ (ii): Assume that $F$ is topologically expansive. By Definition~\ref{exp-uexp}, there exists $k \in \N$ such that
\[
\sup_{n \in \N} \|F^ne_0\|_k = \infty \ \ \ \ \text{or} \ \ \ \ \sup_{n \in \N} \|F^{-n}e_0\|_k = \infty,
\]
which gives (ii).

\smallskip\noindent
(ii) $\Rightarrow$ (iii): First, assume that there exists $k \in \N$ such that
\begin{equation}\label{a1}
\sup\limits_{n \in \N} \|e_{n}\|_{k} = \infty.
\end{equation}
Since $(e_n)_{n \in \Z}$ is a basis of $X$, for any $x\!:= (x_n)_{n \in \Z} \in X$, the sequence $(x_n e_n)_{n \in \Z}$
converges to $0$ in $X$ as $n \to \pm \infty$. Hence, by the Banach–Steinhaus theorem, the family of operators $T_n$ on $X$,
$n \in \Z$, defined by $x \mapsto x_n e_n$, is equicontinuous. Hence, there are $C > 0$ and $\ell \in \N$ such that
\begin{equation}\label{foreq}
\|T_n x\|_k \leq C \|x\|_\ell \ \ \text{ for all } x \in X \text{ and } n \in \Z.
\end{equation}
Let $x\!:= (x_n)_{n \in \Z}$ be any nonzero vector in $X$ and choose $j \in \Z$ such that $x_j \neq 0$. Then
\[
F^m x = (x_{n-m})_{n \in \Z} = \sum_{n \in \Z} x_{n-m} e_n \ \ \text{ for all } m \in \N.
\]
By \eqref{foreq}, we get
\[
\|x_j e_{m+j}\|_k = \|T_{m+j}(F^m x)\|_k \leq C \|F^m x\|_\ell \ \ \text{ for all } m \in \N.
\]
In view of (\ref{a1}), we conclude that
\[
\sup_{m \in \N} \|F^m x\|_\ell = \infty.
\]
Hence, $F$ is topologically positively expansive.
Analogously, $\sup\limits_{n \in \N}\| e_{-n}\|_{k} = \infty$ implies that $F^{-1}$ is topologically positively expansive.

\smallskip\noindent
(iii) $\Rightarrow$ (i): Trivial.
\end{proof}

Now we will see that, using a suitable conjugacy, this result can be generalized to bilateral weighted forward shifts.
Let $F_w : X \to X$ be an invertible bilateral weighted forward shift on a Fréchet sequence space $X$
and define the sequence of weights $v\!:= (v_n)_{n \in \Z}$ by
\[
v_n\!:= \left(\prod_{\nu=n}^{0} w_\nu\right)^{-1} \ \ \ \textnormal{for} \ \ n \leq 0, \ \ \
v_1\!:= 1, \ \ \ v_n\!:= \prod_{\nu=1}^{n-1} w_\nu \ \ \ \textnormal{for} \ \ n \geq 2.
\]
By setting
\[
X_v\!:=\{(x_n)_{n \in \Z} \in \omega(\Z) : (v_n x_n)_{n \in \Z} \in X\},
\]
the mapping $\phi_v : X_v \to X$, given by $(x_n)_{n \in \Z} \mapsto (v_n x_n)_{n \in \Z}$, is an algebraic isomorphism.
If, in addition, we consider the topology of $X_v$ given by
\[
\textrm{ a subset } U \textrm{ of } X_v  \textrm{ is open if and only if } \phi_v(U) \textrm{ is open in } X,
\]
then $X$ becomes a Fréchet sequence space.
If $(e_n)_{n \in \Z}$ is a basis in $X$, then it is also a basis in $X_v$, and the following diagram commutes:
\[
\xymatrix{X_v\ar[r]^{F}\ar[d]_{\phi_v}&X_v\ar[d]^{\phi_v}\\
X\ar[r]_{F_w}&X}
\]
Hence, $F$ and $F_w$ are conjugate operators. Moreover, $F$ and $F_w$ are simultaneously invertible operators.
The previous construction is similar to the construction of \cite[pp. 96, 101]{KGroAPer11}.
Combining these facts and Theorem \ref{shift_space}, we easily obtain the result below.

\begin{theorem}\label{shift_space_w}
Suppose that $X$ is a Fréchet sequence space over $\Z$ in which the sequence $(e_n)_{n \in \Z}$ of canonical vectors is a basis,
$(\|\cdot\|_k)_{k \in \N}$ is an increasing sequence of seminorms that induces the topology of $X$,
and the bilateral weighted forward shift $F_w$ is an invertible operator on $X$. Then the following assertions are equivalent:
\begin{itemize}
\item [\rm (i)] $F_w : X \to X$ is topologically expansive;
\item [\rm (ii)] there exists $k \in \N$ such that
    \begin{itemize}
    \item [\rm (a)] $\sup\limits_{n \in \N} |w_1 \cdot \hdots \cdot w_n| \|e_{n+1}\|_{k} = \infty$ or
    \item [\rm (b)] $\sup\limits_{n \in \N} |w_{-n+1} \cdot \hdots \cdot w_0|^{-1} \|e_{-n+1}\|_{k} = \infty$;
    \end{itemize}
\item [\rm (iii)]
    \begin{itemize}
    \item [\rm (a)] $F_w : X \to X$ is topologically positively expansive or
    \item [\rm (b)] $F_w^{-1} : X \to X$ is topologically positively expansive.
    \end{itemize}
\end{itemize}
\end{theorem}

\begin{remark}
(a) Note that, in Theorems \ref{shift_space} and \ref{shift_space_w}, (ii.a) is equivalent to (iii.a), and (ii.b) is equivalent to (iii.b).
In par\-ti\-cu\-lar, if $X$ is a Fréchet sequence space over $\N$, then a similar formulation of the above result gives a characterization
for topologically positively expansive weighted forward shifts.

\smallskip\noindent
{\rm (b)} The study of topological expansiveness for invertible bilateral weighted backward shifts can be reduced to the
corresponding case of forward shifts.
In this case, $B_w$ is topologically expansive if and only if $B_w$ or $B_w^{-1}$ is topologically positively expansive,
and this is equivalent to the existence of $k \in \N$ such that
\[
\sup_{n \in \N} |w_{-n+2} \cdot \hdots \cdot w_0w_1| \|e_{-n+1}\|_{k} = \infty \  \ \ \textnormal{or} \ \ \
\sup_{n \in \N} |w_{2} \cdot \hdots \cdot w_{n+1}|^{-1} \|e_{n+1}\|_{k} = \infty,
\]
respectively.

\smallskip\noindent
{\rm (c)} If $X$ is a Fréchet sequence space over $\N$ which contains the sequence $a\!:=(1,0,0,\dots)$,
then the unilateral weighted backward shift $B_w$ is not topologically positively expansive, since $B_w a = 0$.
\end{remark}

\begin{remark}
If $T$ is an invertible operator on a locally convex space $X$, it is clear that
\[
T \text{ or } T^{-1} \text{ topologically positively expansive} \ \Rightarrow \ T \text{ topologically expansive}.
\]
As seen above, the converse holds for bilateral weighted shifts.
Nevertheless, the converse is not true in general, even in the Banach space setting, as observed in \cite[Remark~33]{BerCirDarMesPuj18}.
In fact, any invertible hyperbolic operator with nontrivial hyperbolic splitting provides a counterexample.
\end{remark}

An infinite matrix $A\!:= (a_{j,k})_{j,k \in \N}$ is said to be a {\em K\"othe matrix} if $0 \leq a_{j,k} \leq a_{j,k+1}$ for all $j,k \in \N$,
and for each $j \in \N$, there exists $k \in \N$ with $a_{j,k} > 0$. For $1 \leq p < \infty$, consider
\[
\lambda_p(A)\!:= \left\{x\!:= (x_j)_{j=1}^\infty \in \omega(\N) : \|x\|_k\!:= \left(\sum_{j=1}^\infty |x_j a_{j,k}|^p\right)^{1/p} < \infty
\text{ for all } k \in \N\right\},
\]
and, for $p = 0$,
\[
\lambda_0(A)\!:= \left\{x\!:= (x_j)_{j=1}^\infty \in \omega(\N) : \lim_{j \to \infty} x_j a_{j,k}=0, \
\|x\|_k\!:= \sup_{j \in \N} |x_j| a_{j,k} \text{ for all } k \in \N\right\},
\]
which are the corresponding K\"othe echelon spaces (see \cite{HJar81,GKot69,RMeiDVog97}).
These spaces constitute a natural class of Fréchet sequence spaces (whose topology is given by the corresponding sequence
$(\|\cdot\|_k)_{k \in \N}$ of seminorms) in which many authors have studied the dynamical properties of weighted shifts
(see, for instance, \cite{FMarAPer02,XWuPZhu13} and references therein).
In a natural way, the K\"othe spaces can be extended to sequence spaces indexed over $\Z$.
In this setting, $A\!:= (a_{j,k})_{j \in \Z,k \in \N}$ is a {\em K\"othe matrix on $\Z$} if $0 \leq a_{j,k} \leq a_{j,k+1}$ for all $j \in \Z$
and $k \in \N$, and for each $j \in \Z$, there exists $k \in \N$ with $a_{j,k} > 0$.

The spaces $\lambda_p(A,\Z)$ and $\lambda_0(A,\Z)$ are defined as in the case of $\N$,
with the sums and the supremum taken over the $\Z$ (see \cite[Section 3.3]{FMarAPer02}).
A simple computation shows that $(e_n)_{n \in \Z}$ is a basis of $\lambda_p(A,\Z)$.
Moreover, the bilateral weighted forward shift $F_w$ defines an operator on $\lambda_p(A,\Z)$ if and only if
\begin{equation}\label{fwkothe}
\forall \ k \in \N, \ \exists \ m \in \N, \ m > k : \ \sup_{j \in \Z} |w_{j}| \frac{a_{j+1,k}}{a_{j,m}} < \infty,
\end{equation}
and $F_w$ is invertible if and only if
\begin{equation}\label{fw-1kothe}
\forall \ k \in \N, \ \exists \ m \in \N, \ m > k : \ \sup_{j \in \Z} \frac{a_{j,k}}{|w_{j}|a_{j+1,m}} < \infty,
\end{equation}
where $a_{j,m} = 0$ in \eqref{fwkothe} implies $a_{j+1,k} = 0$, and $a_{j+1,m} = 0$ in \eqref{fw-1kothe} implies $a_{j,k}=0$.
In these cases we adopt $0/0$ as $1$.
By applying Theorem~\ref{shift_space_w} to the bilateral weighted shifts on the K\"othe spaces, we deduce the following result.

\begin{corollary}\label{applkothe}
Let $A$ be a K\"othe matrix on $\Z$, $w$ be a weight sequence satisfying \eqref{fwkothe} and \eqref{fw-1kothe},
and $1 \leq p < \infty$ or $p = 0$. Then the following assertions are equivalent:
\begin{itemize}
\item [\rm (i)] $F_w : \lambda_p(A,\Z) \to \lambda_p(A,\Z)$ is topologically expansive;
\item [\rm (ii)] there exists $k \in \N$ such that
    \begin{itemize}
    \item [\rm (a)] $\sup\limits_{n \in \N} |w_1 \cdot \hdots \cdot w_n| a_{n+1,k} = \infty$ or
    \item [\rm (b)] $\sup\limits_{n \in \N} |w_{-n+1} \cdot \hdots \cdot w_0|^{-1} a_{-n+1,k} = \infty$;
    \end{itemize}
\end{itemize}
\end{corollary}

\begin{example}\label{omegaZ}
Let $w\!:= (w_n)_{n \in \Z}$ be a weight sequence and $1 \leq p < \infty$ or $p = 0$.
\begin{itemize}
\item [\rm (i)] If $a_{j,k} = 1$ for all $j \in \Z$ and $k \in \N$, then
    \[
    \lambda_p(A,\Z) = \ell^p(\Z) \ \ \ \ (\text{resp. } \lambda_0(A,\Z) = c_0(\Z)).
    \]
    If, in addition,
    \[
    \sup_{j \in \Z} |w_j| < \infty \ \ \text{ and } \ \ \inf_{j \in \Z} |w_j| > 0,
    \]
    which means that $F_w$ is a well-defined invertible operator on $\ell^p(\Z)$ (resp.\ on $c_0(\Z)$),
    then from Corollary~\ref{applkothe} we recover \cite[Theorem E(1)]{BerCirDarMesPuj18}.

\smallskip
\item [\rm (ii)] If the matrix $A$ is such that for each $k \in \N$, there is $j_k \in \Z$ satisfying $a_{j,k}$ = 0 for all $|j| > j_k$, then
    \[
    \lambda_p(A,\Z) = \omega(\Z).
    \]
    Furthermore, the bilateral weighted shifts on $\omega(\Z)$ are never topologically expansive.
    Indeed, consider a weighted forward shift $F_w$ on $\omega(\Z)$. Since, for each $k \in \N$, we can choose $j_k \in \Z$ such that
    \[
    \sup_{n \in \N} |w_1 \cdot \hdots \cdot w_n| a_{n+1,k} = \sup_{1 \leq n \leq |j_k|-1} |w_1 \cdot \hdots \cdot w_n| a_{n+1,k} < \infty
    \]
    and
    \[
    \sup_{n \in \N} |w_{-n+1} \cdot \hdots \cdot w_0|^{-1} a_{-n+1,k}
    = \sup_{1 \leq n \leq |j_k|+1}| w_{-n+1} \cdot \hdots \cdot w_0|^{-1} a_{-n+1,k} < \infty,
    \]
    the result follows from Corollary \ref{applkothe}.

\smallskip
\item [\rm (iii)] If $a_{j,k} = (|j|+1)^k$ for all $j \in \Z$ and $k \in \N$, then
    \[
    \lambda_p(A,\Z) = \lambda_1(A,\Z) = s(\Z),
    \]
    where $s(\Z)$ denotes the {\em space of rapidly decreasing sequences} on $\Z$.
    Assume that $\sup_{j \in \Z} |w_j| < \infty$ and $\inf_{j \in \Z} |w_j| > 0$. Then clearly $F_w$ defines an invertible operator on $s(\Z)$.
    Although these conditions are sufficient to $F_w$ be invertible, they are not necessary, since if $w_j = \sqrt{|j|+1}$,
    then $F_w$ also defines an invertible operator on $s(\Z)$.
    Under these assumptions we obtain, from Corollary \ref{applkothe}, the following equivalences:
    \begin{itemize}
    \item [\rm (I)] $F_w : s(\Z) \to s(\Z)$ is topologically expansive;
    \item [\rm (II)] there exists $k \in \N$ such that
        \begin{itemize}
        \item [\rm (a)] $\sup\limits_{n \in \N} |w_1 \cdot \hdots \cdot w_n|n^{k} = \infty$ or
        \item [\rm (b)] $\sup\limits_{n \in \N} |w_{-n} \cdot \hdots \cdot w_{-1}|^{-1}n^{k} = \infty$;
        \end{itemize}
    \end{itemize}

\smallskip
\item [\rm (iv)] Let $l,r>0$. If $w_j = (|j|+1)^r$ for every $j \geq 1$ and $w_j = (|j|+1)^{-l}$ for every $j \leq 0$,
    it follows from Corollary \ref{applkothe} that both operators $F_w$ and $F_w^{-1}$ on $s(\Z)$ are topologically
    positively expansive.

\smallskip
\item [\rm (v)] If $v = (v_j)_{j \in \Z}$ is a positive weight sequence and $a_{j,k} = v_j$ for all $j \in \Z$ and $k \in \N$, then
    \[
    \lambda_p(A,\Z) = \ell^p(\Z,v) \ \ \ \ (\text{resp. } \lambda_0(A,\Z) = c_0(\Z, v)).
    \]
    If, in addition,
    \[
    \sup_{j \in \Z} \frac{|w_j|v_{j+1}}{v_j} < \infty \ \ \ \text{ and } \ \ \ \sup_{j \in \Z} \frac{v_j}{|w_j|v_{j+1}} < \infty,
    \]
    which means that $F_w$ is a well-defined invertible operator on $\ell^p(\Z,v)$ (resp.\ on $c_0(\Z,v)$),
    then Corollary~\ref{applkothe} implies that the topologically positively expansive bilateral weighted forward shifts are
    characterized by
    \[
    \sup\limits_{n \in \N} |w_1 \cdot \hdots \cdot w_n| v_{n+1} = \infty.
    \]
\end{itemize}
\end{example}

As we promised after the statement of Corollary~\ref{h-ute}, we will now exhibit some examples of uniformly topologically expansive operators on $s(\Z)$ that are not hyperbolic.

\begin{example}
Consider $s(\Z)$ as being $\lambda_1(A,\Z)$ (Example~\ref{omegaZ}(iii)), that is, consider $s(\Z)$ endowed with the seminorms
\[
\|x\|_k\!:= \sum_{j=-\infty}^{\infty} (|j|+1)^k |x_j|, \ \ \ x\!:= (x_j)_{j \in \Z} \in s(\Z), \ k \in \N.
\]
Given $a > 1$, consider the weight sequence
\[
w\!:= (\ldots,a^{-1},a^{-1},a^{-1},a,a,a,\ldots),
\]
where the first $a$ appears at position $0$. Then $F_w$ is an invertible operator on $s(\Z)$.
Let $P_M : s(\Z) \to M$ and $P_N : s(\Z) \to N$ denote the canonical projections associated to the topological direct sum decomposition
$s(\Z) = M \oplus N$, where
\[
M\!:= \{x \in s(\Z) : x_j = 0 \text{ for all } j \leq -1\} \ \text{ and } \
N\!:= \{x \in s(\Z) : x_j = 0 \text{ for all } j \geq 0\}.
\]
For each $k \in \N$, let
\[
A_k\!:= \{x \in S_{\|\cdot\|_k} : \|P_M x\|_k \geq 1/2\} \ \text{ and } \
B_k\!:= \{x \in S_{\|\cdot\|_k} : \|P_N x\|_k \geq 1/2\}.
\]
Clearly, $S_{\|\cdot\|_k} = A_k \cup B_k$. Moreover, for each $n \in \N$,
\[
\|F_w^n x\|_k \geq \|F_w^n P_M x\|_k \geq a^n \|P_M x\|_k \geq a^n/2 \ \text{ for all } x \in A_k
\]
and, similarly, $\|F_w^{-n} x\|_k \geq a^n/2$ for all $x \in B_k$.
This proves that $F_w$ is uniformly topologically expansive.
On the other hand, since $\|F_w^n x\|_k \to +\infty$ as $n \to \pm \infty$, for every nonzero $x \in s(\Z)$,
$F_w$ is not hyperbolic.
\end{example}

By choosing $a \in (0,1)$ in the above example, we can fulfill the promise we made at the end of Section 2 and exhibit some examples of generalized hyperbolic operators on $s(\Z)$ that are not hyperbolic.

\begin{example}
Given $a \in (0,1)$, let $w$, $M$ and $N$ be as in the previous example.
For each $k \in \N$, define
\[
c_k\!:= \sup_{n \in \N_ 0}\, [(n+1)^k (\sqrt{a})^n] < \infty.
\]
For any $y \in M$ and $n \in \N_0$,
\begin{align*}
\|F_w^n y\|_k &= (n+1)^k a^n |y_0| + (n+2)^k a^n |y_1| + (n+3)^k a^n |y_2| + \cdots\\
  &= a^n \Big((n+1)^k |y_0| + \frac{(n+2)^k}{2^k}\, 2^k |y_1| + \frac{(n+3)^k}{3^k}\, 3^k |y_2| + \cdots\Big)\\
  &\leq (n+1)^k a^n \|y\|_k\\
  &\leq c_k (\sqrt{a})^n \|y\|_k.
\end{align*}
Analogously, for any $z \in N$ and $n \in \N_0$,
\[
\|F_w^{-n} z\|_k \leq c_k (\sqrt{a})^n \|z\|_k.
\]
Thus, $F_w$ is generalized hyperbolic.
However, $F_w$ is not hyperbolic because it is not topologically expansive. Indeed, for any $k \in \N$, since
\[
\|F_w^n e_0\|_k \to 0 \text{ as } n \to \pm \infty,
\]
we have that $\sup_{n \in \Z} \|F_w^n e_0\|_k < \infty$.
Another way to see that $F_w$ is not topologically expansive is to note that the weight sequence $w$ does not satisfy condition~(II) in Example~\ref{omegaZ}(iii).
\end{example}


\section{Some open problems}\label{FinalSection}

If someone asks for an example of an operator with the shadowing property on a Banach space $X$,
one would certainly mention an operator such as $2 I_X$ or $\frac{1}{2} I_X$.
The fact that these operators do not have the shadowing property when $X = H(\C)$ is somehow surprising
and suggests the following question.

\medskip\noindent
\textbf{ Problem A.} Does every Fr\'echet space (or locally convex space) support an operator with the shadowing property?
How about the Fr\'echet space $H(\C)$?

\medskip
Similarly, the simplest examples of expansive operators on a Banach space $X$ are operators such as $2 I_X$ of $\frac{1}{2} I_X$.
But these operators are not expansive when $X = \K^\Z$.

\medskip\noindent
\textbf{ Problem B.} Does every Fr\'echet space (or locally convex space) support an expansive operator?
How about the Fr\'echet space $\K^\Z$?
(We emphasize that {\em expansive} is meant in the metric sense.)

\medskip
It was observed in Remark~\ref{NotShad} that the hypothesis of completeness of the seminorms is essential for the validity
of Theorem~\ref{GHSP}, but we do not know if the other technical hypothesis in Theorem~\ref{GHSP} can be removed or not.
In other words, we have the following question.

\medskip\noindent
\textbf{ Problem C.} Let $X$ be a locally convex space whose topology is induced by a directed family of {\em complete} seminorms.
Is it true that generalized hyperbolicity implies shadowing for operators on $X$?

\medskip
We proved that every invertible generalized hyperbolic operator on a Banach space is topologically stable (Corollary~\ref{GH-TS-Cor}).
However, in the case of an arbitrary sequentially complete locally convex space, we proved that invertible generalized hyperbolic
operators have a stability property which seems to be pretty close to topological stability (Theorem~\ref{GH-TS}), but we had to require
the additional hypothesis that $T - S$ is a bounded map (see (\ref{GH-TS-1})) in order to be able to apply Lemma~\ref{Step1}.
So, we have the following question.

\medskip\noindent
\textbf{ Problem D.} Can we remove the additional hypothesis that $T - S$ is a bounded map from (\ref{GH-TS-1})?
In other words, is every invertible generalized hyperbolic operator on a  non-normable sequentially complete locally convex space
topologically stable? How about the case of non-normable Fr\'echet spaces?

\medskip
Topologically expansive weighted shifts on Fr\'echet sequence spaces were characterized in Theorem~\ref{shift_space}.

\medskip\noindent
\textbf{ Problem E.} Characterize the concept of uniform topological expansivity for weighted shifts on Fr\'echet sequence spaces.

\medskip
As a final problem, let us mention that the following basic questions are still open.

\medskip\noindent
\textbf{ Problem F.} Let $T \in GL(X)$, where $X$ is a Banach space.
\begin{itemize}
\item If $T$ has the shadowing property, is $T$ necessarily generalized hyperbolic?
\item If $T$ is structurally stable, is $T$ necessarily generalized hyperbolic?
\end{itemize}


\bigskip

\section*{Acknowledgements}

The first author is beneficiary of a grant within the framework of the grants for the retraining, modality Mar\'ia Zambrano,
in the Spanish university system (Spanish Ministry of Universities, financed by the European Union, NextGenerationEU)
and was also partially supported by CNPq, Project {\#}308238/2021-4, and by CAPES, Finance Code 001.
The fifth and the first authors were partially supported by the Spanish Ministery MCIN/AEI/10.13039/501100011033/FEDER, UE, Project  PID2022-139449NB-I00.
The fifth and the third authors were partially supported by Generalitat Valenciana, Project PROMETEU/2021/070.
The fourth author was partially supported by FAPEMIG Grants RED-00133-21 and APQ-01853-23.
We would like to thank the referee whose careful review resulted in an improved presentation of the article.



\smallskip

{\footnotesize

\bigskip\noindent
{\sc Nilson C. Bernardes Jr.}

\smallskip\noindent
Institut Universitari de Matem\`atica Pura i Aplicada, Universitat Polit\`ecnica de Val\`encia, Edifici 8E, Acces F, 4a Planta,
46022 Val\`encia, Spain, \ and\\
Departamento de Matem\'atica Aplicada, Universidade Federal do Rio de Janeiro, Caixa Postal 68530, RJ 21945-970, Brazil.\\
\textit{ e-mail address}: ncbernardesjr@gmail.com

\bigskip\noindent {\sc Blas M. Caraballo}

\smallskip\noindent
ESAP, Escuela Superior de Administraci\'on P\'ublica, Cartagena de Indias, Bol\'ivar, Colombia.\\
\textit{ e-mail address:} mbcaraballo07@gmail.com

\bigskip\noindent
{\sc Udayan B. Darji}

\smallskip\noindent
Department of Mathematics, University of Louisville, Louisville, KY 40208-2772, USA.\\
\textit{ e-mail address:} ubdarj01@louisville.edu

\bigskip\noindent
{\sc Vin\'icius V. F\'avaro}

\smallskip\noindent
Faculdade de Matem\'atica, Universidade Federal de Uberl\^andia, Uberl\^andia, MG 38400-902, Brazil.\\
\textit{ e-mail address:} vvfavaro@gmail.com

\bigskip\noindent
{\sc Alfred Peris}

\smallskip\noindent
Institut Universitari de Matem\`atica Pura i Aplicada, Universitat Polit\`ecnica de Val\`encia, Edifici 8E, Acces F, 4a Planta, 46022 Val\`encia, Spain.\\
\textit{ e-mail address}: aperis@mat.upv.es

}


\begin{thebibliography}{99}

\bibitem{AlvBerMes21} F. F. Alves, N. C. Bernardes Jr. and A. Messaoudi,
    \textit{ Chain recurrence and average shadowing in dynamics}, Monatsh.\ Math.\ \textbf{ 196} (2021), no.\ 4, 665--697.

\bibitem{DAno67} D. V. Anosov,
    \textit{ Geodesic flows on closed Riemannian manifolds of negative curvature} (Russian), Trudy Mat.\ Inst.\ Steklov.\ \textbf{ 90} (1967), 209 pp.

\bibitem{AntManVar22} M. B. Antunes, G. E. Mantovani and R. Var\~ao,
    {\it Chain recurrence and positive shadowing in linear dynamics}, J. Math.\ Anal.\ Appl.\ \textbf{ 506} (2022), no.\ 1, Paper No.\ 125622, 17 pp.

\bibitem{NAokKHir94} N. Aoki and K. Hiraide,
    \textit{ Topological Theory of Dynamical Systems - Recent Advances}, North-Holland, Amsterdam, 1994.

\bibitem{TAra18} T. Arai,
    \textit{ Devaney's and Li-Yorke's chaos in uniform spaces}, J. Dyn.\ Control Syst.\ \textbf{ 24} (2018), no.\ 1, 93--100.

\bibitem{BerBonMarPer11} T. Berm\'udez, A. Bonilla, F. Mart\'inez-Gim\'enez and A. Peris,
    \textit{ Li-Yorke and distributionally chaotic operators}, J. Math.\ Anal.\ Appl.\ \textbf{ 373} (2011), no.\ 1, 83--93.

\bibitem{BerBonMulPer15} N. C. Bernardes Jr., A. Bonilla, V. M\"uller and A. Peris,
    \textit{ Li-Yorke chaos in linear dynamics}, Ergodic Theory Dynam.\ Systems \textbf{ 35} (2015), no.\ 6, 1723--1745.

\bibitem{BerCirDarMesPuj18} N. C. Bernardes Jr., P. R. Cirilo, U. B. Darji, A. Messaoudi and E. R. Pujals,
    \textit{ Expansivity and shadowing in linear dynamics}, J. Math.\ Anal.\ Appl.\ \textbf{ 461} (2018), no.\ 1, 796--816.

\bibitem{NBerAMesArxiv} N. C. Bernardes Jr.\ and A. Messaoudi,
    \textit{ Shadowing and structural stability in linear dynamical systems}, arXiv:1902.04386v1 (2019), 20 pp.

\bibitem{NBerAMes21} N. C. Bernardes Jr.\ and A. Messaoudi,
    \textit{ Shadowing and structural stability for operators}, Ergodic Theory Dynam.\ Systems \textbf{ 41} (2021), no.\ 4, 961--980.

\bibitem{NBerAMes20} N. C. Bernardes Jr.\ and A. Messaoudi,
    \textit{ A generalized Grobman-Hartman theorem}, Proc.\ Amer.\ Math.\ Soc.\ \textbf{ 148} (2020), no.\ 10, 4351--4360.

\bibitem{NBerAPerTA} N. C. Bernardes Jr.\ and A. Peris,
    \textit{ On shadowing and chain recurrence in linear dynamics}, Adv.\ Math.\ {\bf 441} (2024), Paper No.\ 109539, 46 pp.

\bibitem{GBir29} G. D. Birkhoff,
    \textit{ D\'emonstration d'un th\'eor\`eme \'el\'ementaire sur les fonctions enti\`eres},
    C. R. Acad.\ Sci.\ Paris \textbf{ 189} (1929), 473--475.

\bibitem{Bo00} J. Bonet,
	\textit{  Hypercyclic and chaotic convolution operators}, J. London Math.\ Soc.\ (2) \textbf{ 62} (2000), no.\ 1, 253--262.
	
\bibitem{BoDo12} J. Bonet and P. Doma\'{n}ski,
	\textit{ Hypercyclic composition operators on spaces of real analytic functions},
	Math.\ Proc.\ Cambridge Philos.\ Soc.\ \textbf{ 153} (2012), no.\ 3, 489--503.

\bibitem{BoFrPeWe2005} J. Bonet, L. Frerick, A. Peris and J. Wengenroth,
	\textit{Transitive and hypercyclic operators on locally convex spaces}, Bull.\ London Math.\ Soc.\ \textbf{ 37} (2005), no.\ 2, 254--264.

\bibitem{BoKalPe2021} J. Bonet, T. Kalmes  and A. Peris,
	\textit{Dynamics of shift operators on non-metrizable sequence spaces}, Rev.\ Mat.\ Iberoam.\ \textbf{ 37} (2021), no.\ 6, 2373--2397.

\bibitem{NBou1989} N. Bourbaki,
    \textit{ General Topology, Chapters 1-4}, Springer-Verlag, Berlin, 1989.

\bibitem{RBow75} R. Bowen,
    \textit{ $\omega$-limit sets for axiom A diffeomorphisms}, J. Differential Equations \textbf{ 18} (1975), no.\ 2, 333--339.

\bibitem{MBriGStu02} M. Brin and G. Stuck,
    \textit{ Introduction to Dynamical Systems}, Cambridge University Press, 2002.
    
\bibitem{BCarVFav20} B. M. Caraballo and V. V. F\'avaro,
    \textit{ Chaos for convolution operators on the space of entire funcions of infinitely many complex variables},
    Bull.\ Soc.\ Math.\ France \textbf{ 148} (2020), no.\ 2, 237--251.

\bibitem{CirGolPuj21} P. R. Cirilo, B. Gollobit and E. R. Pujals,
    \textit{ Dynamics of generalized hyperbolic linear operators}, Adv.\ Math.\ \textbf{ 387} (2021), Paper No.\ 107830, 37 pp.

\bibitem{DoKa18} P. Doma\'{n}ski and C.-D. Kar{\i}ks{\i}z,
	\textit{  Eigenvalues and dynamical properties of weighted backward shifts on the space of real analytic functions},
	Studia Math.\ \textbf{ 242} (2018), no.\ 1, 57--78.

\bibitem{MEisJHed70} M. Eisenberg and J. H. Hedlund,
    \textit{ Expansive automorphisms of Banach spaces}, Pacific J. Math.\ \textbf{ 34} (1970), no.\ 3, 647--656.

\bibitem{LFor1954} L. R. Ford, Jr.,
    \textit{ Homeomorphism groups and coset spaces}, Trans.\ Amer.\ Math.\ Soc.\ \textbf{ 77} (1954), 490--497.

\bibitem{GEPe10} K.-G. Grosse-Erdmann and A. Peris,
	\textit{ Weakly mixing operators on topological vector spaces},
	Rev. R. Acad.\ Cienc.\ Exactas F\'is.\ Nat.\ Ser.\ A Mat.\ RACSAM \textbf{ 104} (2010), no.\ 2, 413--426.

\bibitem{KGroAPer11} K.-G. Grosse-Erdmann and A. Peris Manguillot,
    \textit{  Linear Chaos}, Springer, Berlin, 2011.

\bibitem{PHar60} P. Hartman,
    \textit{ A lemma in the theory of structural stability of differential equations}, Proc.\ Amer.\ Math.\ Soc.\ \textbf{ 11} (1960), 610--620.

\bibitem{JHed71} J. H. Hedlund,
    \textit{ Expansive automorphims of Banach spaces, II}, Pacific J. Math.\ \textbf{ 36} (1971), no.\ 3, 671--675.

\bibitem{HJar81} H. Jarchow,
    \textit{ Locally Convex Spaces}, B. G. Teubner, Stuttgart, 1981.

\bibitem{NKaw2019} N. Kawaguchi,
    \textit{ Topological stability and shadowing of zero-dimensional dynamical systems},
    Discrete Contin.\ Dyn.\ Syst.\ \textbf{ 39} (2019), no.\ 5, 2743--2761.

\bibitem{PKos05} P. Ko\'scielniak,
    \textit{ On genericity of shadowing and periodic shadowing property}, J. Math.\ Anal.\ Appl.\ \textbf{ 310} (2005), no.\ 1, 188--196.

\bibitem{GKot69} G. K\"othe,
    \textit{ Topological Vector Spaces, I}, Die Grundlehren der mathematischen Wissenschaften, Band 159,
    Springer-Verlag, Berlin - Heidelberg - New York, 1969.

\bibitem{GMac52} G. R. MacLane,
    \textit{ Sequences of derivatives and normal families}, J. Analyse Math.\ \textbf{ 2} (1952), no.\ 1, 72--87.

\bibitem{MMaz00} M. Mazur,
    \textit{ Hyperbolicity, expansivity and shadowing for the class of normal operators}, Funct.\ Differ.\ Equ.\ \textbf{ 7} (2000), no.\ 1-2, 147--156.

\bibitem{FMarAPer02} F. Mart\'inez-Gim\'enez and A. Peris,
    \textit{ Chaos for backward shift operators}, Internat.\ J. Bifur.\ Chaos Appl.\ Sci.\ Engrg.\ \textbf{ 12} (2002), no.\ 8, 1703--1715.

\bibitem{RMeiDVog97} R. Meise and D. Vogt,
    \textit{ Introduction to Functional Analysis}, Oxford Graduate Texts in Mathematics, 2,
    The Clarendon Press, Oxford University Press, Oxford - New York, 1997.

\bibitem{AMor81} A. Morimoto,
    \textit{ Some stabilities of group automorphisms}, in: Manifolds and Lie Groups,  Progr.\ Math.\ 14, Birkhäuser, 1981, 283--299.

\bibitem{JOmb94} J. Ombach,
    \textit{ The shadowing lemma in the linear case}, Univ.\ Iagel.\ Acta Math.\ \textbf{ 31} (1994), 69--74.

\bibitem{OsiPilTik10} A. V. Osipov, S. Yu.\ Pilyugin and S. B. Tikhomirov,
    \textit{ Periodic shadowing and $\Omega$-stability}, Regul.\ Chaotic Dyn.\ \textbf{ 15} (2010), no.\ 2-3, 404--417.

\bibitem{JPal68} J. Palis,
    \textit{ On the local structure of hyperbolic points in Banach spaces}, An.\ Acad.\ Brasil.\ Ci.\ \textbf{ 40} (1968), 263--266.

\bibitem{KPal00} K. Palmer,
    {\it Shadowing in Dynamical Systems - Theory and Applications},
    Mathematics and its Applications, vol.\ {\bf 501}, Kluwer Academic Publishers, Dordrecht, 2000.
    
\bibitem{Peris18} A. Peris,
	\textit{ A hypercyclicity criterion for non-metrizable topological vector spaces},
	Funct.\ Approx.\ Comment.\ Math.\ \textbf{ 59} (2018), no.\ 2, 279--284.

\bibitem{SPil99} S. Yu.\ Pilyugin,
    \textit{ Shadowing in Dynamical Systems}, Lecture Notes in Math., 1706, Springer-Verlag, Berlin, 1999.

\bibitem{CPug69} C. C. Pugh,
    \textit{ On a theorem of P. Hartman}, Amer.\ J. Math.\ \textbf{ 91} (1969), 363--367.

\bibitem{JRob72} J. W. Robbin,
    \textit{ Topological conjugacy and structural stability for discrete dynamical systems}, Bull.\ Amer.\ Math.\ Soc.\ \textbf{ 78} (1972), no.\ 6, 923--952.

\bibitem{Shkarin2012} S. Shkarin,
	\textit{  Hypercyclic operators on topological vector spaces}, J. Lond.\ Math.\ Soc.\ (2) \textbf{ 86} (2012), no.\ 1, 195--213.

\bibitem{WSie1920} W. Sierpi\'nski,
    \textit{ Sur une propri\'et\'e topologique des ensembles d\'enombrables denses en soi}, Fund.\ Math.\ \textbf{ 1} (1920), 11--16.

\bibitem{JSin72} Ja.\ G. Sina$\breve{\text{\i}}$,
    \textit{ Gibbs measures in ergodic theory} (Russian), Uspehi Mat.\ Nauk \textbf{ 27} (1972), no.\ 4(166), 21--64.

\bibitem{WUtz50} W. R. Utz,
    \textit{ Unstable homeomorphisms}, Proc.\ Amer.\ Math.\ Soc.\ \textbf{ 1} (1950), 769--774.

\bibitem{PWal1970} P. Walters,
    \textit{ Anosov diffeomorphisms are topologically stable}, Topology \textbf{ 9} (1970), 71--78.

\bibitem{PWal1978} P. Walters,
    \textit{ On the pseudo-orbit tracing property and its relationship to stability},
    Lecture Notes in Math., 668, Springer-Verlag, Berlin-New York, 1978, pp.\ 231--244.

\bibitem{XWuPZhu13} X. Wu and P. Zhu,
    \textit{ Li-Yorke chaos of backward shift operators on K\"othe sequence spaces}, Topology Appl.\ \textbf{ 160} (2013), no.\ 7, 924--929.

\end{thebibliography}
\end{document}